\documentclass[12pt,oneside,reqno]{amsart}
\usepackage{color}
\usepackage{amsmath,amssymb,leftidx,fixltx2e, eucal,mathrsfs,geometry,amsthm}

\usepackage{bm}

\usepackage[english]{babel}
\usepackage{natbib}

\usepackage[T1]{fontenc}

\usepackage{mathtools}

\DeclarePairedDelimiterX{\norm}[1]{\lVert}{\rVert}{#1}

\allowdisplaybreaks

 \usepackage{a4wide}
	
\usepackage{xfrac}

\usepackage{tikz, tikz-cd}
\usepackage{lipsum}

\usepackage{adjustbox}
\usepackage{verbatim}
\usepackage{graphicx}
\usepackage{enumerate}
\usepackage{dcolumn,array}
\usepackage{mathtools, caption} 
\usepackage[cmtip,all]{xy}
\usepackage{upref}
\usepackage{relsize}
\usetikzlibrary{arrows}

\usepackage{xcolor}

\usepackage{tikz-cd}

\tikzset{node distance=1.5cm, auto}

\DeclareMathOperator*{\spn}{span}

\DeclareMathOperator*{\clspn}{\overline{\spn}}

\theoremstyle{plain}
\newtheorem{theorem}{Theorem}[section]
\newtheorem{corollary}[theorem]{Corollary}
\newtheorem{lemma}[theorem]{Lemma}
\newtheorem{proposition}[theorem]{Proposition}
\theoremstyle{definition}
\newtheorem{definition}[theorem]{Definition}

\newtheorem{example}[theorem]{Example}
\newtheorem{remark}[theorem]{Remark}

\numberwithin{equation}{section}

\newcommand{\<}{\langle}
\renewcommand{\>}{\rangle}
\newcommand{\pre}[1]{{}_{#1}}

\newcommand{\tn}{\textnormal}

\renewcommand{\ker}{\operatorname{Ker}}

\newcommand{\id}{\operatorname{id}}
\newcommand{\ad}{\operatorname{Ad}}

\newcommand{\fun}{\mathcal E}

\newcommand{\AXA}{$\pre AX_A$}

\newcommand{\cat}{\ensuremath{\mathsf{ECCor}}}
\newcommand{\AMB}{$\pre AM_B$}

\newcommand{\AXC}{$\pre AX_C$}
\newcommand{\BYD}{$\pre BY_D$}

\newcommand{\AXB}{$\pre AX_B$}
\newcommand{\ANB}{$\pre AN_B$}

\newcommand{\BYC}{$\pre BY_C$}

\newcommand{\AAA}{$\pre AA_A$}

\newcommand{\BNC}{$\pre BN_C$}

\newcommand{\BYB}{$\pre BY_B$}

\newcommand{\CZC}{$\pre CZ_C$}

\newcommand{\C}{$C^*$}
\newcommand{\ots}{\otimes_A}
\newcommand{\otss}{\otimes_B}

\newcommand{\KK}{\mathcal K}
\newcommand{\LL}{\mathcal L}

\newcommand{\NN}{\mathbb N}
\newcommand{\ZZ}{\mathbb Z}
\newcommand{\TT}{\mathbb T}
\newcommand{\XX}{\mathcal X}

\allowdisplaybreaks

\newcommand{\norms}{\norm[\Big]}

\newcommand{\OX}{\mathcal{O}_X}
\newcommand{\TX}{\mathcal{T}_X}

\newcommand{\OY}{\mathcal{O}_Y}

\newcommand{\OZ}{\mathcal{O}_Z}

\newcommand{\On}{\OY^n}

\newcommand{\tsum}{\textstyle\sum}

\newcommand{\core}{\mathscr B_{[0,\infty)}}

\newcommand{\cores}{\mathscr B_{[1,\infty)}}

\newcommand{\DRD}{$\pre DR_D$}

\newcommand{\ARB}{$\pre AR_B$}
\newcommand{\BSA}{$\pre BS_A$}
\newcommand{\BNA}{$\pre BN_A$}

\newcommand{\CKD}{$\pre CK_D$}

\newcommand\myeq{\stackrel{\mathclap{\normalfont\mbox{*}}}{=}}

\newcommand\myothereq{\stackrel{\mathclap{\normalfont\mbox{**}}}{=}}

\renewcommand{\subset}{\subseteq}

\newcommand{\otsb}{\otimes_{\core}}
\newcommand{\ot}{\otimes}

\newcommand{\B}{\mathscr B}

\newcommand{\otsc}{\otimes_C}

\newcommand{\ec}{\emph{C}}

\newcommand{\oy}{\otimes_{\OY}}

\newcommand{\X}{\widetilde{X}}

\newcommand{\OXX}{\mathcal O_{\X}}

\newcommand{\OXXX}{\mathcal O_{\XX}}

\renewcommand{\zeta}{\eta}

\newcommand{\ibm}{imprimitivity bimodule}
\newcommand{\hbm}{Hilbert bimodule}
\newcommand{\righttext}[1]{\qquad\text{#1 }}
\newcommand{\midtext}[1]{\qquad\text{#1 }\qquad}

\def\paragraph#1{{\bf #1\ }}

\newcommand{\oba}{\pre {A}}

\newcommand{\ma}{m_A}
\newcommand{\mb}{m_B}
\newcommand{\ily}{i_{l,y}}
\newcommand{\ilx}{i_{l,x}}

\newcommand{\irx}{i_{r,x}}

\newcommand{\rz}{1_R\ot\gamma_z}

\renewcommand{\zeta}{\eta}

\renewcommand{\cite}{\citep}
\renewcommand{\restriction}[1]{|_{#1}}

\makeatletter
\@namedef{subjclassname@2020}{%
  \textup{2020} Mathematics Subject Classification}
\makeatother

\usepackage[nice]{nicefrac}

\begin{document}
\title{Passing \C-correspondence Relations to the Cuntz-Pimsner algebras}
\author{M. Ery\"uzl\"u}
\address{Department of Mathematics, University of Colorado Boulder, Boulder, CO 80309-0395}
\email{Menevse.Eryuzlu@colorado.edu}

\date{\today}




\begin{abstract}
We construct a functor that maps \C-correspondences to their Cuntz-Pimsner algebras. Applications include a  generalization of the well-known result of Muhly and Solel:  Morita equivalent \C-correspondences have Morita equivalent Cuntz-Pimsner algebras; as well as the result of Muhly, Pask, and Tomforde: regular strong shift equivalent \C-correspondences have Morita equivalent Cuntz-Pimsner algebras.

. 
 
\end{abstract}

\maketitle

\section{Introduction}

When the generalization of Cuntz-Pimsner algebras was fully completed by Katsura \citep{katsura}, two natural questions arose.

\begin{itemize}
\item Can Morita theory be investigated in the \C-correspondence setting? 
\item If there is a certain relation between two \C-correspondences, what can be said about their Cuntz-Pimsner algebras? 
\end{itemize}

In 1998, Muhly and Solel presented a ground breaking work on the investigation of Morita theory in the \C-correspondence setting \citep{MS}. They developed a notion of Morita equivalence for given \C-correspondences \AXA\ and \BYB , where there is an \ibm\ \AMB\ so that
\[ X\ots M \cong M\otss Y \] as $A-B$ correspondences. They proved that if two \emph{injective} \C-correspondences are Morita equivalent  then the corresponding Cuntz-Pimsner algebras are Morita equivalent in the sense of Rieffel. In \citep{EKK}, the authors presented an elegant  proof of a generalized result where they drop the assumption of injectivity. In 2008, Muhly, Pask and Tomforde \citep{MPT} introduced a weaker relation between \C-correspondences: \emph{strong shift equivalence}. They proved that regular strong shift equivalent correspondences have Morita equivalent Cuntz-Pimsner algebras. Our motivation was to construct a method that significantly shortens the proofs of these results as well as allows us to easily determine the \ibm\ between the associated Cuntz-Pimsner algebras.

To study the relation between \C-correspondences and their Cuntz-Pimser algebras, we define a functor from a category we call \cat\ to the enchilada category. We constract \cat\ so that two objects being isomorphic is equivalent to them being Morita equivalent as \C-correspondences. The enchilada category has \C-algebras as objects, and isomorphism classes of \C-correspondences as morphisms. Our functor maps a given \C-correspondence \AXA\ to its Cuntz-Pimsner algebra $\OX$, and a morphism from \AXA\ to \BYB\ is mapped to the isomorphism class of an $\OX-\OY$ correspondence.

One of our initial hurdles was to assign an $\OX-\OY$ correspondence to a given morphism \AXA\ $\rightarrow$ \BYB\ in \cat\ . We overcome this issue by defining an \emph{injective}  covariant $(\pi,\Phi)$ representation for \AXA\ (Proposition \ref{rep}).  We prove that the representation $(\pi,\Phi)$  admits a gauge action (Proposition \ref{gauge}). Then the Gauge Invariant Uniqueness Theorem gives us one of our main results:  $\OX$ is isomorphic to the \C-algebra generated by the representation $(\pi,\Phi)$.

In Section 6 we use our techniques to prove the results regarding Morita equivalent and strong shift equivalent correspondences as we aimed. There is in fact more accomplished in that section: we prove that the Morita equivalence between the Cuntz-Pimsner algebras associated to the regular strong shift equivalent correspondences is  gauge equivariant (Theorem \ref{equi}).

This paper serves as an updated version of \cite{failure2}. The previous definition of \cat  , unfortunately, yielded a subtle gap in the proof of \cite[Theorem 5.1]{failure2}: the functor might not be well-defined. We conquer this problem by modifying the morphisms in \cat , which does not affect any of our main results but one application. In fact, it is still an open problem whether \cite[Theorem 6.7]{failure2} holds.


\emph{Acknowledgment:} The author is grateful to Toke Carlsen for pointing out the gap in \cite{failure2}, and to Adam Rennie, Alexander Mundey and John Quigg for valuable discussions.

\section{Prelimineries}

A \C-correspondence \AXB\ is a right Hilbert $B$-module equipped with a left action given by a homomorphism $\varphi_X: A\rightarrow \LL(X),$ where $\LL(X)$ denotes the \C-algebra of adjointable operators on $X$. The correspondence \AXB\ is called \emph{non-degenerate} if the set $A\cdot X = \{\varphi_X(a)x: a\in A, x\in X \}$ is dense in $X$.  Note here that by Cohen-Hewitt factorization theorem we have $\overline{A\cdot X}=A\cdot X.$ 
In this paper \emph{all our correspondences will be non-degenerate by standing hypothesis.}

A \C-correspondence \AXB\ is called \emph{injective} if the left action $\varphi_X: A\rightarrow \LL(X)$ is injective; it is called \emph{regular} if the homomorphism $\varphi_X$ is injective  and $\varphi_{X}(A)$ is contained in the \C-algebra $\KK(X)$ of compact operators on $X$.  

A \C-correspondence homomorphism is a triple $(\Phi, \varphi_l, \varphi_r)$: \AXC\ $\rightarrow$ \BYD\ consist of a linear map $\Phi: X\rightarrow Y$, and homomorphisms $\varphi_l: A\rightarrow B$ and $\varphi_r: C\rightarrow D$  satisying
\begin{enumerate}[(i)]
\item $\Phi(a\cdot x)=\varphi_l(a)\cdot \Phi(x)$,
\item $\varphi_r( \<x,z\>_C ) = \<\Phi(x), \Phi(z)\>_{D}$,
\end{enumerate}
for all $a\in A$, and $x,z\in X.$

The triple $(\Phi, \varphi_l, \varphi_r)$ is called  a \emph{\C-correspondence isomorphism} if, in addition, $\Phi$ is bijective and $\varphi_l, \varphi_r$ are isomorphisms. In this paper we mostly deal with the situations where $A=B$ and $C=D$. In those cases we just take $\varphi_l$ and $\varphi_r$  to be the identity maps on $A$ and $C$ respectively, and we simply denote the isomorphism \AXC $\rightarrow$ $\pre AY_C$ by $\Phi$ instead of the triple $(\Phi, \id_A , \id_C)$.  We denote by $\ad\Phi: \LL(X)\rightarrow \LL(Y)$ the associated \C-algebra isomorphism.

Let $A_0$ and $B_0$ be dense $*$-algebras of \C-algebras $A$ and $B$, respectively. An $A_0-B_0$ bimodule $X_0$ is called a \emph{pre-correspondence} if it has a $B_0$-valued semi-inner product satisfying 
\[ \<x, y\cdot b\> = \<x,y\> b ,   \midtext{} \<x,y\>^*=\<y,x\> \]
and $\<a\cdot x, a\cdot x\> \leq \norm{a}^2\<x,x\>$ for all $a\in A_0, b\in B_0$ and $x,y\in X_0$. The Hausdorff completion $X$ of $X_0$ becomes an  $A-B$ correspondence by taking the limits of the operations.

\begin{proposition}\label{pre}\tn{\citep[Lemma 1.23]{enchilada}}
Let $X_0$ be an $A_0-B_0$ pre-correspondence, and let  $Z$ be an $A-B$ correspondence. If there is a map $\Phi : X_0 \rightarrow Z$ satisfying 
\[ \Phi (a\cdot x) = \varphi_Z (a) \Phi(x)  \midtext{and} \<\Phi(x), \Phi(y) \>_B = \<x, y\>_{B_0} , \]
for all $a\in A_0$ and $x,y\in X_0$, then $\Phi$ extends uniquely to an injective $A-B$ correspondence homomorphism $\tilde{\Phi}: X \rightarrow Z . $
\end{proposition}

The \emph{balanced tensor product} $X\otimes_BY$ of
an $A-B$ correspondence $X$ and a $B-C$ correspondence $Y$ is
formed as follows:
the algebraic tensor product $X\odot Y$
is a pre-correspondence with the $A-C$ bimodule structure satisfying
\[
a(x\otimes y)c=ax\otimes yc
\righttext{for}a\in A,x\in X,y\in Y,c\in C,
\]
and the unique $C$-valued semi-inner product whose values on elementary tensors are given by
\[
\<x\otimes y,u\otimes v\>_C=\<y,\<x,u\>_B\cdot v\>_C
\righttext{for}x,u\in X,y,v\in Y.
\]

This semi-inner product defines a $C$-valued inner product on the quotient of $X\odot Y$ by the subspace generated by elements of form 
\[ x\cdot b \otimes y - x\otimes \varphi_Y(b)y \righttext{($x\in X$, $y\in Y$, $b\in B$)}.\]
The  completion, i.e., the Hausdorff completion of $X\odot Y$, is an $A-C$ correspondence $X\otimes_B Y$, where the left action is given by 
 \[A\rightarrow \LL(X\otss Y),  \righttext{$a\mapsto \varphi_X(a)\ot 1_Y,$}\]
for $a\in A.$ 

We denote the canonical image of $x\otimes y$ in $X\otss Y$ by $x\otss y$. The term \emph{balanced} refers to the property
\[
x\cdot b\otss y=x\otss b\cdot y
\righttext{for}x\in X,b\in B,y\in Y,
\]
which is automatically satisfied.

\begin{lemma}[\cite{fowler}]\label{compacts} Let $X$ be a \C-correspondence over $A$ and let $I$ be an ideal of $A$. Then we have the following.
\begin{enumerate}[\normalfont(1)]
\item There is an isometric embedding $\iota: \KK(XI)\rightarrow \KK(X)$ such that 
\[\theta_{\xi, \nu}\mapsto \theta_{\xi, \nu} \midtext{} \tn{for $\xi,$ $\nu \in XI$}.\] 
Moreover, for $T\in \KK(XI)$,  the operator $\iota(T)$ is the unique extension of $T$ to an operator in $\LL(X)$ whose range is contained in $XI.$ 
\item Assume $Y$ is an $A-B$ correspondence. If the left action $\varphi_Y: A\rightarrow \LL(Y)$ is injective, the map 
$\iota: T \mapsto T\ot 1_Y$  gives an isometric homomorphism of $\LL(X)$ into $\LL(X\ots Y).$ If, in addition, $\varphi_Y(A)\subset\KK(Y)$, then $\iota$ embeds $\KK(X)$ into $\KK(X\ots Y).$ 
\end{enumerate}
\end{lemma}

A \emph{Hilbert bimodule} \AXB\ is a \C-correspondence
that is also equipped with an $A$-valued inner product $\pre A\<\cdot,\cdot\>$,
which satisfies
\[
\pre A\<ax,y\>=a\pre A\<x,y\>\midtext{and}\pre A\<x,y\>^*=\pre A\<y,x\>
\]
for all $a\in A,x,y\in X$,
as well as the compatibility property
\[
\pre A\<x,y\>z=x\<y,z\>_B
\righttext{for}x,y,z\in X.
\]
A \hbm\ $\pre AX_B$ is \emph{left-full} if the closed span of $\pre A\<X,X\>$ is all of $A$. 

An \emph{\ibm} \AXB\ is an \hbm\ that is full on both the left and the right. It's dual $\pre B\tilde{X}_A$ is formed as follows: write $\tilde{x}$ when a vector $x\in X$ is regarded as belonging to $\tilde{X}$, define $B-A$ bimodule structure by 
\[ b\tilde{x}a=\widetilde{a^*xb^*}\] 
and the inner product by
\[ \pre B\<\tilde{x},\tilde{y}\>=\<x,y\>_B \midtext{and} \<\tilde{x},\tilde{y}\>_A=\pre A\<x,y\>\]
for $a\in A, b\in B,$ and $x,y\in X. $

The \emph{identity correspondence} on $A$ is the Hilbert bimodule $\pre AA_A$ where bimodule structure given by multiplication, and the inner products are given by 
\[ \oba\<a,b\>=ab^*, \hspace{.5cm} \<a,b\>_A= a^*b, \righttext{for $a,b\in A$}.\]

\begin{lemma}\label{ibm} Let \AXB\ be an \ibm\ and $\pre B\tilde{X}_A$ be it's dual. Then, the maps
\begin{align*}
&m_A: X\otss \tilde{X} \rightarrow A, &\righttext{$ x_1\otss \tilde{x}_2\mapsto \pre A\<x_1, x_2\>$}\\
&m_B: \tilde{X} \ots X \rightarrow B, &\righttext{$\tilde{x}_1\ots x_2 \mapsto \<x_1, x_2\>_B$}
\end{align*}
are \C-correspondence isomorphism satisfying the equality
\[ m_A(x\otss \tilde{y})\cdot z = x\cdot m_B(\tilde{y}\ots x)\]
for all $x,y,z\in X$. 
\end{lemma}


A \emph{representation} $(\pi,t)$ of \AXA\ on a \C-algebra $B$ consists of a $*-$homomorphism $\pi: A\rightarrow B$ and a linear map $t: X\rightarrow B $ such that 
\[
\pi(a)t(x)=t(\varphi_X(a)(x)) \midtext{and} t(x)^* t(y)=\pi(\<x,y\>_A), 
\]
for $a\in A$ and $x, y\in X$, where $\varphi_X$ is the left action homomorphism associated with \AXA. An application of the \C-identity shows that $t(x)\pi(a)=t(x \cdot a)$ is also valid.  For each representation $(\pi,t)$ of \AXA\ on $B$,  there exist a homomorphism $\Psi_t: \KK(X)\rightarrow B$ such that 
\[
\Psi_t (\theta_{x,y})=t(x){t(y)}^*
\]
for $x,y\in X. $ The representation $(\pi,t)$ is called \emph{injective} if $\pi$ is injective, in which case $t$ is an isometry and $\Psi_t$ is injective. We denote the \C-algebra generated by the images of $\pi$ and $t$ in $B$ by \C$(\pi, t).$ 

\begin{lemma}\tn{\citep[Lemma 2.4]{katsura}} Let $(\pi, t)$ be a representation of a given \C-correspondence \AXA\  . Then we have 
\begin{enumerate}[\normalfont(i)]
\item $\pi(a)\Psi_t(k)=\Psi_t(\varphi_X(a)k)$
\item $\Psi_t(k)t(x)=t(kx)$
\end{enumerate}
for $a\in A$, $x\in X$, and $k\in\KK(X)$. 
\end{lemma}

 Now, consider a \C-correspondence \AXA . Let  $X^{\otimes 0}=A,$  $X^{\otimes 1}=X$, and for $n\geq 2$ let $X^{\otimes n}=X\ots X^{\otimes (n-1)}.$ Each $X^{\otimes n}$ is a \C-correspondence over $A$ with
\begin{align*}
\varphi_n (a)(x_1\ots x_2\ots ......\ots x_n) &:=\varphi_X(a)x_1 \ots x_2 \ots...\ots x_n \\
(x_1\ots x_2\ots......\ots x_n)\cdot a &:= x_1\ots x_2\ots......\ots (x_n\cdot a).
\end{align*}
The operator $\varphi_0(a)\in \KK(X^{\otimes 0})$ is just the left multiplication operator on $A$.  For any given representation $(\pi,t)$ of \AXA\ on a \C-algebra $B$,  set $t^0=\pi$ and $t^1=t$. For $n\geq 2$, define a linear map 
\[ t^n : X^{\otimes n} \rightarrow B, \midtext{} t^n(x\ots y)=t(x)t^{n-1}(y),\]
where $x\in X$, $y\in X^{\otimes n-1}$. Then $(\pi, t^n)$ is a representation of $\pre AX^{\otimes n}_A$ on $B$. The associated homomorphism $\Psi_{t^n}: \KK(X^{\otimes n})\rightarrow B$ is given by 
\[ \Psi_{t^n}(\theta_{\xi,\nu}) = t^n(\xi)t^n(\mu)^*,\]
for $\xi, \mu \in X^{\otimes n}$. If $(\pi, t)$ is injective, then the linear map $t^n$ is isometric, and  the homomorphism $\Psi_{t^n}$ is injective. 

For a representation $(\pi, t)$ of \AXA\ we have  \citep[Proposition 2.7]{katsura},
\[ C^*(\pi, t)= \tn{ $\clspn\{t^n(x_n)t^m(y_m)^*:$ $x_n\in X^{\otimes n}, y_m\in X^{\otimes m},$ $n,m\geq 0$\}}. \]

Consider a \C-correspondence \AXA . The ideal $J_X$ is defined as
\begin{align*}
 J_X &= \varphi_X^{-1}(\KK(X))\cap (\ker\varphi_X)^{\perp}\\
 &=\text{\{$a\in A$ : $\varphi_X(a)\in\KK(X)$ and $ab=0$ for all $b\in\ker\varphi_X$}\},
\end{align*}
and is called the \emph{Katsura ideal}. The ideal $J_X$ is the largest ideal of $A$ such that the restriction map $J_X \rightarrow \LL(X)$ is an injection into $\KK(X).$ Notice here that if \AXA\ is regular, i.e, the left action $\varphi_X: A\rightarrow\LL(X)$ is injective and $\varphi_X(A)\subset\KK(X)$, then $J_X=A.$

A representation $(\pi, t)$ of \AXA\ is called \emph{covariant} if $\pi(a)=\Psi_{t} (\varphi_X(a) )$, for all $a\in J_X .$

The \C-algebra generated by the universal representation  of \AXA\  is called the \emph{Toeplitz algebra} $\TX$ of \AXA  . The \C-algebra generated by the universal covariant representation  of \AXA\ is called the \emph{Cuntz-Pimsner} algebra $\OX$ of \AXA .

\section{Categories}

As mentioned in the introduction, in the enchilada category our objects are $C^*$-algebras, and the morphisms from $A$ to $B$ are the isomorphism classes of $A-B$ correspondences. The composition of [\AXB]: $A\rightarrow B$ with [\BYC]: $B\rightarrow C$ is the isomorphism class of the balanced tensor product $\pre A(X\otss Y)_C$; the identity morphism on $A$ is the isomorphism class of the identity correspondence $\pre AA_A$, and the zero morphism $A\rightarrow B$ is $[\pre A0_B].$  It is a crucial fact for this work that a morphism \tn{[\AXB]} is an isomorphism in the enchilada category if and only if \AXB\ is an \ibm\ \cite[Lemma 2.4]{enchilada}. A detailed study of the enchilada category can be found in \cite{taco}.

\underline{\emph{Notation:}}  For any given \C-correspondence \AMB\  we denote the correspondence isomorphism
\[M\ots A \rightarrow M, \righttext{ $m\ots a \mapsto m\cdot a$}\]
by $\iota_{r,m}$, and the correspondence  isomorphism
\[A\ots M\rightarrow M, \righttext{ $a\ots m \mapsto a\cdot m$}\]
by $\iota_{l,m}.$

\begin{definition} Let \AXA , \BYB , and \AMB\ be \C-correspondences, and let $U_M: X\ots M \rightarrow M\otss Y$  be an $A-B$  correspondence isomorphism. Then the isomorphism class of the pair (\AMB , $U_M$), denoted by $[\pre AM_B , U_M]$, consists of pairs (\ANB, $U_N$) such that
\begin{itemize}
\item there exists an isomorphism  $\xi$: \AMB\ $\rightarrow$ \ANB;
\item $U_N: X\ots N \rightarrow N\otss Y$ is a \C-correspondence isomorphism; and 
\item the diagram
\[
\begin{tikzcd}
X\ots M \arrow{r}{1\ot \xi} \arrow[swap]{d}{U_M} & X\ots N \arrow{d}{U_N} \\
M\otss Y \arrow{r}\arrow{r}{\xi\ot 1_Y} & N\otss Y
\end{tikzcd}
\]
commutes.
\end{itemize}
\end{definition}

We now introduce our domain category:

\begin{theorem}

There exists a category \cat\ such that 
\begin{itemize}
\item objects are \C-correspondences;
\item morphisms \AXA $\rightarrow$ \BYB\ are isomorphism classes of the pairs \tn{(\AMB , $U_M$)} where $U_M$ denotes an $A-B$ correspondence isomorphism $X\ots M \rightarrow M\otss Y$, and \AMB\ is a regular correspondence satisfying $J_X\cdot M \subset M\cdot J_Y$;
\item the composition \tn{[\BNC, $U_N$]$\circ$[\AMB , $U_M$]} is given by the isomorphism class \[\tn{[$\pre A(M\otss N)_C$, $U_{M\otss N}$]} \] 
where $U_{M\otss N}$ denotes the isomorphism $(1_M\ot U_N)(U_M\ot 1_N)$;
\item the identity morphism on \AXA\ is \tn{[\AAA , $U_A$]}, where $U_A$ denotes the isomorphism $\ilx^{-1}\circ\irx: X\ots A \rightarrow A\ots X$.
\end{itemize}
\end{theorem}

\begin{proof} Let [\AMB , $U_M$] $\in$ Mor(\AXA , \BYB) and [\BNC, $U_N$] $\in$ Mor(\BYB , \CZC). Then, it is not difficult to verify that [$\pre A(M\otss N)_C$, $(1_N\ot U_M)(U_M\ot 1_N)$] is in Mor(\AXA , \CZC). Now, let [\CKD, $U_K$] $\in$ Mor(\CZC , \DRD). The composition is associative:
\begin{align*}
&\big([\pre CK_D , U_K] \circ [\pre BN_C, U_N]\big)\circ [\pre AM_B, U_M]\\
&\hspace{1cm}=\big[\pre B(N\otsc K)_D, (1_N\ot U_K)(U_N\ot 1_K)\big]\circ \big[\pre AM_B, U_M\big]\\
&\hspace{1cm}=\left[\pre A(M\otss(N\otsc K))_D, \big(1_M\ot (1_N \ot U_K)(U_N \ot 1_K)\big)\big(U_M\ot 1_{N\otsc K}\big)\right]\\
&\hspace{1cm}=\left[\pre A((M\otss N)\otsc K)_D, \big(1_{M\otss N}\ot U_K\big)\big((1_M\ot U_N)(U_M\ot 1_N)\ot 1_K\big)\right]\\
&\hspace{1cm}= [\pre CK_D , U_K] \circ \big([\pre BN_C, U_N]\circ [\pre AM_B, U_M]\big).
\end{align*}

It remains to prove that [\AAA , $U_A$] $\in$ Mor(\AXA , \AXA) is the identity morphism on \AXA . Let [\ANB , $U_N$] $\in$ Mor(\AXA , \BYB). We show that the following diagram commutes.
\[
\begin{tikzcd}
X\ots A\ots N \arrow{r}{1_X \ot \iota_{l,n}} \arrow[swap]{d}{(1_A \ot U_N)(U_A\ot 1_N)} & X\ots N \arrow{d}{U_N} \\
A\ots N \otss Y \arrow{r}\arrow{r}{\iota_{l,n} \ot 1_Y} & N\otss Y
\end{tikzcd}
\]
Let $x\in X$, $a\in A,$ and $n\in N$. On one hand we have 
\begin{align*}
(\iota_{l,n} \ot 1_Y)(1_A\ot U_N)(\ilx^{-1} \ot 1_N)(a \cdot x \ots n)&=(\iota_{l,n} \ot 1_Y)\left(a\ots U_N (x\ots n)\right)\\
&=a\cdot U_N (x\ots n)\\
&= U_N (a\cdot x\ots n).
\end{align*}
By linearity and density the equalities above hold for any element of $X\ots N$, i.e,  
\begin{equation}\label{1}
(i_{l,n}\ot 1_Y)(1_A\ot U_N)(\ilx^{-1}\ot 1_N)=U_N.
\end{equation}
 On the other hand, 
\begin{align*}
U_N(1_X\ot \iota_{l,n})(x\ots a \ots n)=U_N(x\ots a\cdot n)&=U_N(x\cdot a \ots n)\\
&=U_N(\irx\ot 1_N)(x\ots a\ots n).
\end{align*}
Again by linearity and density, we may conclude that
\begin{equation}\label{2}
U_N(1_X\ot \iota_{l,n})=U_N(\irx\ot 1_N)
\end{equation}
Now, combining  equations \ref{1} and \ref{2} we get 
\begin{align*}
U_N(1_X\ot \iota_{l,n})&= U_N(\irx\ot 1_N)\\
&= (\iota_{l,n} \ot 1_Y)(1_A\ot U_N)(\ilx^{-1}\ot 1_N)(\irx \ot 1_N)\\
&= (\iota_{l, n}\ot 1_Y)(1_A\ot U_N)(U_A\ot 1_N),
\end{align*}
which implies $[\pre AN_B]\circ [\pre AA_A]=\big[\pre A(A\ots N)_B, (1_A\ot U_N)(U_A\ot 1_N)\big] = [\pre AN_B , U_N].$
\end{proof}

\begin{proposition}\label{catiso} A morphism \tn{[\AMB, $U_M$]}: \AXA $\rightarrow$ \BYB\ in \cat\ is an  isomorphism if and only if \AMB\ is an \ibm . 
\end{proposition}

To prove Proposition \ref{catiso} we first need the following lemma.

\begin{lemma}\label{inverse} Let   \AMB\  be an \ibm\ given with an $A-B$ correspondence isomorphism $U_M: X\ots M \rightarrow M\otss Y$. Let \BNA\ be the dual of \AMB . Then, there exists a $B-A$ correspondence isomorphism $U_N: Y\otss N \rightarrow N\ots X$ such that 
\[(\ilx\ot 1_{M})(\ma\ot 1_{X\ots M})(1_{M\otss N}\ot U_M^{-1}) = U_M^{-1}(1_M\ot \ily)(1_M\ot \mb\ot 1_Y)\]
as operators on $M\otss N\ots M\otss Y$, where $m_A: M\otss N\rightarrow A$ and $m_B: N\ots M \rightarrow B$ are the isomorphisms defined in Lemma \ref{ibm}. 
\end{lemma}

\begin{proof} Define a $B-A$ correspondence isomorphism $U_N: Y\otss N \rightarrow N\ots X$ as follows:
\[
\begin{tikzcd}[row sep=huge]
Y\otss N \arrow{r}{\ily^{-1}\ot 1_N} & B\otss Y\otss N \arrow{r}{\mb^{-1}\ot 1_Y\ot 1_N} &N\ots M\otss Y\otss N\arrow{dl}[swap]{1_N\ot U_M^{-1}\ot 1_N}\\
&N\ots X\ots M \otss N \arrow{r}{1_N \ot 1_X \ot \ma} & N\ots X\ots A \arrow{r}{1_N\ot\irx} & N\ots X,
\end{tikzcd}
\]
i.e,
\[ U_N := (1_N\ot \irx)(1_{N\ots X}\ot m_A)(1_N\ot U_M^{-1}\ot 1_N)(m_B^{-1}\ot 1_{Y\otss N})(\ily^{-1}\ot1_N). \]

Now, notice that by linearity and density, it suffices to prove the required equality for the elements of form $\xi\otss n\ots \mu \otss y$, where $\xi,\mu\in M, n\in N,$ and $y\in Y$:
\begin{align*}
&(\ilx\ot 1_M)(\ma \ot 1_{X\ots M})(1_{M\ots N} \ot U_M^{-1})(\xi\otss n\ots \mu \otss y) \\
&\hspace{1cm} = (\ilx\ot 1_M)\left[ m_A(\xi\otss n)\ots U_M^{-1}(\mu\otss y)\right]\\
&\hspace{1cm} = m_A(\xi\otss n)\cdot U_M^{-1}(\mu\otss y)\\
&\hspace{1cm} =  U_M^{-1}\left[m_A(\xi\otss n)\cdot \mu\otss y \right] & \righttext{($U_M^{-1}$ preserves the left action)}\\
&\hspace{1cm} = U_M^{-1}\left[ \xi\cdot m_B(n\ots\mu)\otss y \right] & \righttext{(Lemma \ref{ibm})}\\
&\hspace{1cm} = U_M^{-1}\left[\xi \otss m_B(n\ots\mu)\cdot y \right]\\
&\hspace{1cm} = U_M^{-1}(1_M\ot \ily)(1_M\ot \mb\ot 1_Y)(\xi\otss n\ots \mu \otss y),
\end{align*}
which completes the proof. \end{proof}

\begin{proof} [Proof of Proposition~\ref{catiso}] It is not difficult to see that if [\AMB, $U_M$] is an isomorphism in \cat\  then \AMB\ must be invertible. For the other direction, let \BNA\ and $U_N$ be as in Lemma \ref{inverse}. We show that the diagram
\[
\begin{tikzcd}[row sep=large, column sep=large]
X\ots M\otss N \arrow{r}{1_X\ot \ma} \arrow {d}[swap]{(1_M\ot U_N)(U_M\ot 1_N)}&X\ots A\arrow{d}{U_A} \\
M\otss N\ots X \arrow{r}{\ma\ot1_X} & A\ots X 
\end{tikzcd}
\]
commutes. This will allow us to conclude that $[M\otss N,  (1_M\ot U_N)(U_M\ot 1_N)]=[\pre AA_A, U_A]$, i.e, [\BNA, $U_N$] is a right inverse for [\AMB, $U_M$].

First, observe that the definition of $U_N$ and Lemma \ref{inverse} together gives us the equality
\[(1_M\ot U_N)(U_M\ot 1_N)= (1_M\ot 1_N\ot\irx)(\ma^{-1}\ot 1_X \ot \ma)(\ilx^{-1}\ot 1_M\ot 1_N).\]	
Now, let $x=a\cdot x' \in X$, where $a\in A$, $x'\in X$, and let  $\xi\in M\otss N.$ Then we have 
\begin{align*}
&(\ma\ot 1_X)(1_M\ot 1_N\ot\irx)(\ma^{-1}\ot 1_X \ot \ma)(\ilx^{-1}\ot 1_M\ot 1_N)(x\ots\xi)\\
&=(\ma\ot 1_X)(1_M\ot 1_N\ot\irx)(\ma^{-1}\ot 1_X \ot \ma)(a\ots x'\ots\xi)\\
&=(\ma\ot 1_X)(1_M\ot 1_N\ot\irx)(\ma^{-1}(a)\ots x' \ots \ma(\xi))\\
&=a\ots x'\cdot \ma(\xi).
\end{align*}
On the other hand, we have
\begin{align*}
U_A(1_X\ot \ma)(x\ots \xi)&= \ilx^{-1}\circ\irx(x\ots\ma(\xi))\\
&= \ilx^{-1}(x\cdot\ma(\xi))\\
&=\ilx^{-1}(x)\cdot\ma(\xi)\\
&=a\ots x'\cdot\ma(\xi).
\end{align*}

We have shown that $(\ma \ot 1_X)(1_M\ot U_N)(U_M\ot 1_N) = U_A (1_X\ot \ma)$, as desired. One can use the same technique to show that [\BNA, $U_N$] is also a left inverse for [\AMB, $U_M$].  \end{proof}

\section{A Covariant Representation}

In this section we define an injective covariant representation of a  \C-correspondence \AXA\ given with a morphism [\AMB, $U_M$]: \AXA\ $\rightarrow$ \BYB\ in \cat . We prove that this representation in fact admits a gauge action. Since we use it frequently, we would like to remind the reader that any given Hilbert module isomorphism $U: X_A \rightarrow Y_A$ gives rise to an isomorphism $\ad U: \LL(X) \rightarrow \LL(Y)$ such that $\ad U(T)=UTU^{-1}$ for any $T\in\LL(X). $

Let \AXA , \AMB\ be given, where the latter is a regular \C-correspondence. Consider the linear map 
\[T: X\rightarrow \LL(M, X\ots M), \midtext{} T(x)(m):= x\ots m,\] 
where $x\in X,$ $m\in M$. Then we have
\[\<T(x)m, y\ots m'\>_B = \<x\ots m, y\ots m'\>_B = \<m, \varphi_M(\<x,y\>_A)m' \>_B,\]
for $x,y\in X$ and $m,m'\in M$.  This means the adjoint $T(x)^*$ satisfies \[T(x)^*(y\ots m)=T(x)^*T(y)m= \varphi_M(\<x,y\>_A)m,\] for any elementary tensor $y\ots m \in (X\ots M)$. 

On the other hand, we know by regularity that the homomorphism $\varphi_M: A\rightarrow \LL(M)$  is injective and  $\varphi_M(a)\in\KK(M)$ for any $a\in A.$ This allows us to observe that
\[ T(x)\in \KK(M, X\ots M) \iff T(x)^*T(x)\in \KK(M) \iff \varphi_M(\<x,x\>_A)\in \KK(M), \]
which implies  $T(x)\in \KK(M, X\ots M)$, for any $x\in X.$

\begin{lemma}\label{isos} Let $(\Upsilon, t)$ be the universal covariant representation of \AXA . Then we have the following.
\begin{enumerate}[\normalfont(i)]
\item Consider the subspace \[\overline{t(X)\OX}:=\clspn\{t(x)S: \hspace{.2cm} x\in X, S\in\OX\}\]
of $\OX$. The map $X\ots \OX \rightarrow \overline{t(X)\OX}$ determined on elementary tensors by \[x\otimes S \rightarrow t(x)S\] is an $A-\OX$ correspondence isomorphism. 
\item $J_X\cdot \OX \subset \overline{t(X)\OX} $. 

\item When \AXA\ is regular, the map defined in \tn{(i)} gives an isomorphism \[\pre A(X\ots O_X)_{\OX}\rightarrow \pre A(\OX)_{\OX}.\]

\end{enumerate}
\end{lemma}

\begin{proof}  Let $\Phi: X\odot \OX\rightarrow \overline{t(X)\OX}$ be the unique linear map such that $x\otimes S\mapsto t(x)S.$ It suffices to make our computations with elementary tensors. We first show that $\Phi$ preserves the semi-inner product. Let $x,y\in X$ and $S,T\in\OX$. Then,
\begin{align*}
\< x\ot S, y\ot T\>_{\OX} & = \< S, \Upsilon(\<x,y\>_A)T\>_{\OX}\\
& = \< S, t(x)^*t(y)T\>_{\OX} = S^*t(x)^*t(y)T \\
& = \<t(x)S, t(y)T\>_{\OX}. 
\end{align*}
For $a\in A$, the computation
\[\Phi(a\cdot (x\ot S))= \Phi(ax\ot S) = t(ax)S = \Upsilon(a)t(x)S = a\cdot \Phi(x\ot S)\] 
shows that $\Phi$ preserves the left action. It is clear that $\Phi$ is surjective,  and hence, $\Phi$ extends to a unique \C-correspondence isomorphism $X\ots \OX \rightarrow \overline{t(X)\OX}$. 

Part \tn{(ii)} follows from the fact that for any $a\in J_X$ we have  \[\Upsilon(a)= \Psi_t(\varphi_X(a))\in \Psi_t(\KK(X)) = \clspn\{t(x)t(y)^*: x,y\in X\}. \]

For the last part, let \AXA\ be a regular \C-correspondence. Then, since $A=J_X$, we have 
\[ \OX = A\cdot \OX = J_X\cdot \OX \subset \overline{t(X)\OX} ,\]
hence the map defined in \tn{(i)} is an isomorphism onto $\OX$. 
\end{proof}


\begin{proposition}\label{rep}
Let \tn{[\AMB, $U_M$]:} \AXA $\rightarrow$ \BYB\ be a morphism in \cat . Then \AXA\ has an injective covariant representation on $\KK(M\otss \OY)$. 
\end{proposition}

\begin{proof}
Denote the universal covariant representation of \BYB\ by  $(\Upsilon, t)$. We use the \C-correspondence isomorphisms 
\[ U_M:  X\ots M  \rightarrow M\otss Y \midtext{and} V_Y: Y\otss \OY \rightarrow \overline{t(Y)\OY} \]
to construct a linear map $\Phi: X\rightarrow \KK(M\otss \OY). $
 
For each $x\in X$, define an operator $T(x) : M\rightarrow M\otss Y$ by \begin{equation}\label{T}T(x)m = U_M(x\ots m).\end{equation} Since \AMB\ is regular, as discussed above Lemma \ref{isos}, we have \[T(x)\in \KK(M, M\otss Y).\] In addition, since the \C-correspondence $\pre B{(\OY)}_{\OY}$ is regular, we have  
\[ T(x)\ot 1_{\OY}\in \KK(M\otss \OY, M\otss Y\otss \OY) \midtext{} \text{(Lemma \ref{compacts})}.\]
And now, the operator
\[\Phi(x):=(1_M \ot V_Y)(T(x)\ot 1_{\OY})\in \KK(M\otss \OY, M\otss \overline{t(Y)\OY})\] 
can be viewed as a compact operator on the Hilbert $\OY$-module $M\otss \OY$ whose range is contained in the submodule $M\otss \overline{t(Y)\OY}.$  \footnote{as in the first item of Lemma \ref{compacts}}

Now, define a homomorphism $\pi: A\rightarrow \LL(M\otss \OY)$  by   \[\pi(a)= \varphi_M(a)\ot 1_{\OY},\]
for $a\in A$. Note that $\pi$ is injective and $\pi(A)\subset \KK(M\otss \OY)$, by Lemma \ref{compacts}.

We claim that $(\pi, \Phi)$ is a representation of \AXA . Let $x,x' \in X,$ and $a\in A$. Then we have
\begin{align*}
\Phi(x)^* \Phi(x') &= (T(x)^* \ot 1_{\OY})(1_M\ot V_Y)^*(1_M\ot V_Y) (T(x')\ot 1_{\OY}) \\
 &=  T(x)^*T(x')\otss 1_{\OY}\hspace{.6cm}\text{($1_M\ot V_Y$ is a unitary map)}\\
& = \varphi_M(\<x,x'\>_A)\otss 1_{\OY} \\
&= \pi(\<x,x'\>_A).
\end{align*}
It remains to show the equality $\pi(a)\Phi(x)=\Phi(\varphi_X(a)x)$. Observe that we have 
\[ (\varphi_M(a)\ot 1_{\OY})(1_M\ot V_Y)= (1_M \ot V_Y)(\varphi_M(a)\ot 1_Y\ot 1_{\OY}).\]
This allows us to make the following computation. 
\begin{align*}
\Phi(\varphi_X(a)x) &= (1_M\ot V_Y)(T(a\cdot x)\ot 1_{\OY}) \\
&=(1_M\otss V_Y)(\varphi_M(a)\ot 1_Y\ot 1_{\OY})(T(x)\ot 1_{\OY})\\
&= (\varphi_M(a)\ot 1_{\OY})(1_M\ot V_Y)(T(x)\ot 1_{\OY}) \\
& = (\varphi_M(a)\ot 1_{\OY})\Phi(x)\\
& = \pi(a)\Phi(x).
\end{align*}

We now prove that  the representation $(\pi, \Phi)$ is covariant. Let 
$ \Psi_{\Phi}: \KK(X)\rightarrow \KK(M\otss \OY)$ be the injective homomorphism associated to the representation $(\pi, \Phi)$.  Then, for $x, x' \in X$, we have the following equalities on $M\otss\overline{t(Y)\OY}.$ 
\begin{align*} 
\Psi_{\Phi}(\theta_{x,x'}) &= \Phi(x)\Phi(x')^* \\
&=(1_M\ot V_Y)(T(x)\otss 1_{\OY})(T(x')\ot 1_{\OY})^*(1_M\ot V_Y)^*\\
&=  (1_M\ot V_Y)(T(x)T(x')^* \otss 1_{\OY})(1_M\ot V_Y)^*\\
&= (1_M\ot V_Y)[\ad U_M(\theta_{x,x'}\ot 1_M)\ot 1_{\OY}](1_M\ot V_Y)^*.
\end{align*}
This implies that, for any $k\in\KK(X)$, we have
\[ \Psi_{\Phi}(k)\restriction{M\otss\overline{t(Y)\OY}}= (1_M\ot V_Y)[\ad U_M(k \ot 1_M)\ot 1_{\OY}](1_M\ot V_Y)^*. \] 
In particular,  for $a\in J_X$, we have
\begin{align*}
\Psi_{\Phi}(\varphi_X(a))\restriction{M\otss\overline{t(Y)\OY}} & = (1_M\ot V_Y)[\ad U_M(\varphi_X(a)\ot 1_M)\ot1_{\OY}](1_M\ot V_Y)^* \\
&=  (1_M\ot V_Y)(\varphi_M(a)\ot 1_Y \ot 1_{\OY})(1_M\ot V_Y)^*\\
&= \ad(1_M\ot V_Y)(\varphi_M(a)\ot 1_Y\ot 1_{\OY})\\
&= \varphi_M(a)\ot 1_{\OY}\restriction{M\otss\overline{t(Y)\OY}}\\
&=\pi(a)\restriction{M\otss\overline{t(Y)\OY}}.
\end{align*}
On the other hand, for $a\in J_X$, we know that the image of the operator $\pi(a) \in \KK(M\otss \OY)$ is contained in \[J_XM\otss \OY \subset MJ_Y\otss \OY = M\otss J_Y\OY \subset M\otss \overline{t(Y)\OY} .\] 
Coupling this with the fact that $\Psi_{\Phi}(\varphi_X(a))=\pi(a)$ on $M\otss \overline{t(Y)\OY}$ we get \[\pi(a)^*\pi(a) = \Psi_{\Phi}(\varphi_X(a))^*\pi(a) \in \KK(M\otss\OY).\] In other words, we have 
\[ \pi(a^*a)= \Psi_{\Phi}(\varphi_X(a))^*\pi(a) = \big( \pi(a^*)\Psi_{\Phi}(\varphi_X(a)) \big)^* = \Psi_{\Phi}(\varphi_X(a^*a)).\]
One can now show that 
\[
\norm{  \Psi_{\Phi}(\varphi_X(a))-\pi(a)}^2 = \norm{ (\Psi_{\Phi}(\varphi_X(a))-\pi(a))^*(\Psi_{\Phi}(\varphi_X(a))-\pi(a))}=0,\]
for any $a\in J_X$, which completes the proof.
\end{proof}

\begin{definition}\label{c}\normalfont Let  \tn{[\AMB, $U_M$]}: \AXA\ $\rightarrow$ \BYB\ be a morphism in \cat . Then, the injective covariant representation $(\pi, \Phi)$ of \AXA\ defined as in the proof of Proposition \ref{rep} is called  the \emph{C-covariant representation.} 

\end{definition}

Now, the universality of the Cuntz-Pimsner algebra $\OX$ gives us the following result.

\begin{corollary}

Let \tn{[\AMB, $U_M$]:} \AXA $\rightarrow$ \BYB\ be a morphism in \cat . Let $(\pi, \Phi)$ be the associated C-covariant representation of \AXA . Then, there exists a unique homomorphism $\sigma: \OX\rightarrow \KK(M\otss \OY)$ 
such that  \[\sigma(t_X(x))=\Phi(x) \midtext{and} \sigma(\Upsilon_X(a))=\pi(a), \] for $x\in X$, $a\in A$, where  $(\Upsilon_X, t_X)$ denotes the universal covariant representation of \AXA . Thus, the regular \C-correspondence $\pre A(M\otss \OY)_{\OY}$ can be viewed as an  $\OX-\OY$ correspondence via the homomorphism $\sigma$. 

\end{corollary}

\begin{lemma}\label{id} 
Let $(\Upsilon, t)$ be the universal covariant representation of \AXA , and let  $(\pi,\Phi)$ be the C-covariant representation of \AXA \ associated to the identity morphism $[\pre AA_A, U_A]$: \AXA $\rightarrow$ \AXA\ in \cat . Then, the $A-\OX$ correspondence isomorphism 
\[ U: A\ots\OX \rightarrow \OX, \righttext{ $a\ots S\mapsto \Upsilon(a)S$}\]
preserves the left $\OX$ module structure, i.e., 
\[ \pre {\OX}(A\ots \OX)_{\OX} \cong \pre {\OX}{\OX}_{\OX}.\]
\end{lemma}

\begin{proof}
Let $x\in X$, $S\in\OX$, $a\in A$. Note that $x\cdot a = a'\cdot x' $ for some $a'\in A, x'\in X,$  by Cohen-Hewitt factorization theorem. Now we have 
\begin{align*}
U\big[ t(x)\cdot (a\ots S) \big] = U\big[ \Phi(x)(a\ots S) \big] &= U\big[ a'\ots t(x')S\big] \\
&= \Upsilon(a')t(x')S\\
&= t(x\cdot a)S\\
&= t(x)U(a\ots S),
\end{align*}
which implies, by linearity and density, that $U[t(x)\cdot m ] = t(x)\cdot U(m)$ for any $m\in (A\ots \OX).$ Moreover, one can easily verify that $U[\Upsilon(a)\cdot m] = \Upsilon(a)\cdot U(m),$ for any $a\in A$ and $m\in (A\ots \OX)$. This complete the proof as  elements $t(x)$ and $\Upsilon(a)$ generate $\OX$. 
\end{proof}

\begin{definition}\label{sub}

A \C-correspondence \AXA\ is said to be a nondegenerate subcorrespondence of \BYB\ if there exists an $A-B$ correspondence homomorphism $(\phi, \varphi)$: \AXA $\rightarrow$ \BYB\ such that 
\begin{enumerate}[\normalfont(i)]
\item the linear map $\phi: X\rightarrow Y$ is injective;
\item the homomorphism $\varphi: A\rightarrow B$ is injective and non-degenerate;
\item  $Y=\overline{\phi(X)B}$.
\end{enumerate}
\end{definition}

Notice that any nondegenerate subcorrespondence of an injective correspondence is injective, by definition.

\begin{lemma}\label{XinXB}
 Let  \AXA\ be a nondegenerate subcorrespondence of  a \C-correspondence \BYB . Then, in \cat , there exists a morphism  from \AXA\ to \BYB\  if 
 \begin{equation}\label{j}
 J_X\cdot B\subset J_Y.
 \end{equation} Condition \tn{(\ref{j})} follows immediately when \BYB\ is injective. 
 \end{lemma}

\begin{proof} Let $(\phi, \varphi)$: \AXA $\rightarrow$ \BYB\ be as in Definition \ref{sub}. Then, the homomorphism $\varphi$ induces a regular correspondence $\pre AB_B$. We first show that the unique linear map  \[ \xi : X\odot B \rightarrow Y \midtext{} x\otimes b \mapsto \phi(x)b,\]
for $x\in X$, $b\in B,$ extends to an $A-B$ correspondence isomorphism $X\ots B\rightarrow Y.$ As usual, we make all our computations with elementary tensors, as it suffices. For $x,x'\in X$, $a\in A$, and $b,b'\in B$, we have
\begin{align*}
\xi( a\cdot (x\otimes b) ) = \xi( a\cdot x\otimes b) &= \phi(a\cdot x)b \\
&= \varphi(a)\cdot\phi(x)b & \text{(by definition of $(\phi, \varphi)$)}\\
&=\varphi(a)\cdot\xi(x\otimes b) 
\end{align*}
and 
\begin{align*}
\<\xi(x\otimes b), \xi(x'\otimes b')\>_B = \< \phi(x)b, \phi(x')b'\>_B &= b^*\<\phi(x),\phi(x')\>_B b'  \\
&= b^*\varphi(\<x,x'\>_A)b'\\
&= \<b, \<x, x'\>_A \cdot b'\>_B \\
&= \<x\otimes b, x'\otimes b'\>_B. 
\end{align*}
Thus, by Proposition \ref{pre}, the map $\xi$ extends to an injective $A-B$ correspondence homomorphism, which is clearly surjective.

On the other hand, we have the $A-B$ correspondence isomorphism \[j: B\otss Y\rightarrow Y  \righttext{$b\otss y\mapsto b\cdot y$}.\]
Then, the composition
\begin{equation}\label{ub}U_B := j^{-1}\circ\xi: X\ots B\rightarrow B\otss Y\end{equation} is an $A-B$ correspondence isomorphism. We may now conclude that the isomorphism class $[\pre A B_B, U_B]$: \AXA $\rightarrow$ \BYB\ is a morphism in \cat\ if we are given the condition $J_X\cdot B \subset J_Y$.

We complete the proof by showing that if \BYB\ is injective then  $J_X\cdot B \subset J_Y$: let $a\in J_X$. Since $\pre AB_B$ is regular we have $\varphi_X(a)\ot 1_B\in\KK(X\ots B)$; which implies \[\ad\xi(\varphi_X(a)\ot 1_B)=\varphi_Y(\varphi(a))\in\KK(Y).\] This means $\varphi(a)\in J_Y,$ since \BYB\ is injective. Then for any $b\in B$, we have 
\[ a\cdot b = \varphi(a)b\in J_Y, \]
as desired. 
\end{proof}

\begin{proposition}\label{proplast}

Let \AXA\ and \BYB\ be injective correspondences, and let  $(\phi, \varphi)$: \AXA $\rightarrow$ \BYB\  be as in Definition \ref{sub}. Denote the C-covariant representation of \AXA\ on $\KK(B\otss \OY)$ by $(\pi, \Phi),$ the universal covariant representation of \BYB\ by $(\Upsilon_Y, t_Y)$, and the natural \C-algebra isomorphism $\KK(B\otss \OY)\rightarrow  \OY$ by $\iota$. Then, 
\[ \iota(\Phi(x)) = t_Y(\phi(x)) \midtext{and} \iota(\pi(a)) = \Upsilon_Y(\varphi(a))\]
for all $x\in X$, $a\in A$. In other words, the \C-algebra \C$(\pi,\Phi)$ is isomorphic to the \C-algebra generated by $t_Y(\phi(X))$ and $\Upsilon_Y(\varphi(A))$. 
\end{proposition}

\begin{proof} Note first that the map $\iota$ is really the composition of the isomorphisms $s: \KK(\OY)\rightarrow \OY$ and $\ad \mu: \KK(B\otss \OY)\rightarrow \KK(\OY),$ where $\mu$ denotes the $A-\OY$ module isomorphism $B\otss \OY\rightarrow \OY$ determined on elementary tensors by $\mu(b\otss S)=\Upsilon_Y(b)S,$ for $b\in B,$ $S\in\OY$. We will show that the the following diagram commutes.

\begin{center}
\begin{tikzpicture}[scale=1, transform shape]
\node (K) [scale=0.8]  {$\KK(B\otss\OY)$};
\node (KK) [right of=K, xshift=2cm, scale=0.8] {$\KK(\OY)$};
\node (X) [above of=K, left of=K, scale=0.8] {$X$};
\node(Y)[ right of=X, scale=0.8]{$Y$};
\node(KB)[right of=KK, xshift=2cm, scale=0.8]{$\OY$};
\draw[right hook->] (X) to  node[scale=0.8, above] {$\phi$} (Y);
\draw[->] (X) to node[scale=0.8, left] {$\Phi$} (K);
\draw[->, bend left] (Y) to node[scale=0.8, above] {$t_Y$} (KB);
\draw[->] (K) to node[scale=0.7, above] {$\ad\mu$} (KK);
\draw[->] (KK) to node[scale=0.7, above] {$s$} (KB);
\node(A) [left of=K, below of=K, scale=0.8]{$A$};
\node(B)[right of=A, scale=0.8]{$B$};
\draw[right hook->] (A) to node[scale=0.8, below] {$\varphi$} (B);
\draw[->] (A) to node[scale=0.8, left] {$\pi$} (K);
\draw[->, bend right] (B) to node[scale=0.8, below] {$\Upsilon_Y$} (KB);
\end{tikzpicture}
\end{center}

Let $U_B:=j^{-1}\circ \xi: \pre A(X\ots B)_B \rightarrow \pre A(B\otss Y)_B$ be the isomorphism (\ref{ub}) defined in the proof of Lemma \ref{XinXB}. Let $x\in X, b\in B,$ and $S\in\OY$. Then,  \[\Phi(x)(b\otss S) = (1_B\ot V_Y)(U_B(x\ots b)\otss S)\]
by construction, where $V_Y$ is the unitary map $Y\otss \OY \rightarrow \overline{t_Y(Y)\OY}.$ Note here that we have 
\[U_B(x\ots b) = j^{-1}\circ\xi(x\ots b)=j^{-1}(\phi(x)b)=b'\otss y,\]
for some $b'\in B, y\in Y$ satisfying $\phi(x)b=b'y,$ by Cohen-Hewitt factorization theorem.
This gives us  \[\Phi(x)(b\otss S)= (1_B\ot V_Y)(U_B(x\ots b)\otss S)= b'\otss t_Y(y)S, \]
which implies 
\begin{align*}
\mu\circ \Phi(x)(b\otss S) = \mu(b'\otss t_Y(y)S) &= \Upsilon_Y(b')t_Y(y)S\\
&= t_Y(b' \cdot y)S\\
&= t_Y(\phi(x)b)S\\
&= t_Y(\phi(x))\Upsilon_Y(b)S \\
&= t_Y(\phi(x))\mu(b\otss S)\\
&= s^{-1}\big(t_Y(\phi(x)\big)[\mu(b\otss S)].
\end{align*}
This computation allows one to conclude, by linearity and density, that 
\begin{align*}
 s^{-1}\big(t_Y(\phi(x))\big) &= \mu\circ\Phi(x)\circ\mu^{-1} \\
 &= \ad\mu\big(\Phi(x)\big),
 \end{align*}
 which implies $\iota\big(\Phi(x)\big)=t_Y\big(\phi(x)\big)$, as desired. Lastly, for $a\in A,$ we have
\[ \mu\circ\pi(a)(b\otss S)= \mu(\varphi(a)b\otss S)= \Upsilon_Y(\varphi(a)b)S=\Upsilon_Y(\varphi(a))\mu(b\otss S),\]
which suffices to complete the proof. 
\end{proof}

Recall that a representation $(\pi,t)$ of $X$ \emph{admits a gauge action} if for each $z\in\TT$ there exists a homomorphism $\beta_z : C^*(\pi,t) \rightarrow C^*(\pi,t)$ such that \[\beta_z(\pi(a))=\pi(a) \midtext{and} \beta_z(t(x))=zt(x)\]  for all $a\in A,$ and $x\in X.$ If it exists, the homomorphism $\beta_z$ is unique. The map  \[\beta: \TT\rightarrow Aut\left( C^*(\pi, t) \right), \righttext{ $z \mapsto \beta_z$}\] 
is called the \emph{gauge action}. One can easily show that $\beta$ is a strongly continuous homomorphism.

\begin{theorem}[The Gauge Invariant Uniqueness Theorem]\label{GUT}
Let the  pair $(\Upsilon,t)$ be the universal covariant representation of  \AXA .  Assume $(\phi_X, t_X)$ is an injective covariant representation of \AXA\ on a \C-algebra $B$. If  $(\phi_X, t_X)$ admits a gauge action, then the homomorphism $\rho: \OX\rightarrow B$ is injective. In other words, the  natural surjection $\rho: \OX\rightarrow C^*(\phi_X, t_X)$ is an isomorphism. 

\end{theorem}

A proof of the above theorem can be found in \citep{katsura}.

\begin{remark}\label{gaugeremark}
 Let $\gamma$ be the gauge action for the universal covariant representation $(\Upsilon , t)$ of \BYB . Then, for any $z\in \TT$, we have \[\gamma_z(t^n(y_n))= z^nt^n(y_n) \midtext{and} \gamma_z(t^n(y_n)^*)= z^{-n}t^n(y_n)^* ,\] where $y_n\in Y^{\otimes n}$ . Now, for each $n\in \ZZ$ consider the subspace
  \[ \On := \{T\in \OY  : \gamma_z(T)= z^n(T), \hspace{.2cm} \text{for all $z\in \TT$}\} . \]
We have
\begin{align*}
\OY&=\clspn\{ \tn{$t^n(y_n)t^m(y_m)^*$: $y_n\in Y^{\otimes n}$,  $y_m\in Y^{\otimes m}$, $n,m\geq 0$} \}\\
&=\clspn\{ T_s \in \mathcal{O}_{Y}^{s}: \tn{$s\in\ZZ$} \},
\end{align*}
which implies that elements of form $m\otss T_n$, where $m\in M$ and $T_n\in \On$, densely span $M\otss\OY$. 
\end{remark}


\begin{proposition}\label{gauge}
Let \tn{[\AMB, $U_M$]}: \AXA\ $\rightarrow$ \BYB\ be a morphism in \cat . The associated  C-covariant representation $(\pi, \Phi)$ of \AXA\ admits a gauge action. 
\end{proposition}

\begin{proof}
 Let $\gamma$ be the gauge action for the universal covariant representation $(\Upsilon , t)$ of \BYB , and let $z\in\TT$.  The linear map $1_M\ot \gamma_z : M\odot \OY \rightarrow M\odot \OY$ satisfying \[(1_M\ot \gamma_z) (m\ot S) = m\ot \gamma_z(S)\] for $m\in M$, $S\in \OY$ is bounded. Indeed, let $\sum_i m_i\ot S_i \in M\odot\OY$. 
We have 
\begin{align*}
\norms{  (1_M\ot \gamma_z)\big(\tsum_i m_i\ot S_i\big)}_{\OY}^2 & = \norms{ \tsum_{i,j}\big\<\gamma_z(S_i), \<m_i,m_j\>_B \cdot \gamma_z(S_j)\big\>_{\OY}}\\
& = \norms{ \tsum_{i,j}\gamma_z(S_i)^* \gamma_z(\<m_i,m_j\>_B\cdot S_j)}  \\
& = \norms{ \tsum_{i,j} \gamma_z(S_i^*\<m_i,m_j\>_B\cdot S_j)} \\
& = \norms{ \gamma_z( \tsum_{i,j}(S_i^*\<m_i,m_j\>_B\cdot S_j)} \\
& = \norms{ \tsum_{i,j}S_i^*\<m_i,m_j\>_B\cdot S_j} \\ 
& = \norms { \tsum_i m_i\ot S_i}^2, 
\end{align*}
where the symbol $\norm{\cdot}$ represents the semi-norm on the pre-correspondence $M\odot\OY$. Now by continuity we may conclude that $1_M\ot \gamma_z$ extends to a well-defined bounded linear operator on $M\otss\OY$. However, this operator is not adjointable. This can be easily seen with elementary tensors: let $m,n\in M$ and $T_i\in\OY^i$, $T_j\in\OY^j$. The computation
\begin{align*}
\< (1_M\otimes \gamma_z)(m\otss T_i), n\otss T_j\>_{\OY}&
= \< m\otss \gamma_z(T_i), n\otss T_j\>_{\OY} \\
&= \<z^i T_i, \<m,n\>_B\cdot T_j\>_{\OY}\\
&= \<T_i, z^{-i}\<m,n\>_B\cdot T_j\>_{\OY},
\end{align*}
suffices. 

We claim that the homomorphism $\beta_z: C^*(\pi, \Phi)\rightarrow  C^*(\pi, \Phi)$ defined by 
\[  \beta_z(T) = (1_M\ot \gamma_z)T(1_M\ot \gamma_{\overline{z}}) \]
is a gauge action for the representation $(\pi,\Phi).$ The key point here is that even though  $1_M\ot \gamma_z$ is not an adjointable operator on $M\otss \OY$,  the operator 
\[ (1_M\ot \gamma_z)k(1_M\ot \gamma_{\bar{z}}) \]
\emph{is} adjointable for any $k\in \KK(M\otss\OY)$. It suffices to prove this for $k=\theta_{m_1\otss T_i, m_2\otss T_j}$, where $m_1, m_2\in M$, $T_i\in \OY^i,$ $T_j\in \OY^j$ by Remark \ref{gaugeremark}.  Let $m,n\in M,$ $T_l\in \OY^l$ and $T_k\in \OY^k$. First observe that we have 

\begin{align*}
(1_M\ot \gamma_z)k(1_M\ot \gamma_{\bar{z}})(m\otss T_k)
&= m_1\ot \gamma_z \big(T_i \<m_2\ot T_j, m\otss z^{-k}T_k\>_{\OY}\big) \\
& = m_1\ot\gamma_z \big( T_i T_j^* \<m_2,m\>_B\cdot z^{-k}T_k\big)\\
&= m_1\ot z^{i-j+k} z^{-k} T_i T_j^*\<m_2, m\>_B \cdot T_k .\\
\end{align*}
This allows us to make the following computation.

 \begin{align*}
&\big\<(1_M\ot \gamma_z)k(1_M\ot \gamma_{\bar{z}})(m\otss T_k),n\otss T_l\big\>_{\OY}\\
&\hspace{3cm}= \big\<m_1\otss z^{i-j+k} z^{-k} T_i T_j^*\<m_2, m\>_B \cdot T_k, n\otss T_l\big\>_{\OY} \\
& \hspace{3cm}= \big\<z^{i-j}T_i T_j^*\<m_2,m\>_B\cdot T_k, \<m_1, n\>_B\cdot T_l\big\>_{\OY} \\
&\hspace{3cm}= \big\< T_k, z^{j-i} \<m, m_2\>_B \cdot T_j T_i^* \<m_1, n\>_B\cdot T_l\big\>_{\OY}\\
& \hspace{3cm}= \big\< m\otss T_k, (1_M\ot \gamma_z)k^*(1_M\ot \gamma_{\bar{z}})(n\otss T_l)\big\>_{\OY}.
\end{align*}

Now, let  $z\in\TT$. In order to complete the proof of our claim we need to show 
\[ \beta_z(\Phi(x))\xi = z\Phi(x)\xi \midtext{and} \beta_z(\pi(a))\xi=\pi(a)\xi,\]
for any $\xi\in (M\otss \OY)$, $x\in X$, $a\in A.$ We check the equilaties for the elements of form $m\otss T_n$ as it suffices. A crucial fact here is that for an $m\in M$, $T_n\in \OY^n$, we have $\Phi(x)(m\otss T_n)\in M\otss\OY^{n+1}$ by construction, and thus  \begin{equation}\label{g} (1_M\ot \gamma_z)\Phi(x)(m\otss T_n)= z^{n+1}\Phi(x)(m\otss T_n).\end{equation} This allows us to make the following computation.
\begin{align*}
\beta_z(\Phi(x)) (m\otss T_n) &= (1_M\ot \gamma_z) \Phi(x) (m\otss z^{-n}T_n) \\
&= z^{-n} (1_M\ot\gamma_z)  \Phi(x)(m\otss T_n),\\
&= z^{-n} z^{n+1} \Phi(x)(m\otss T_n) \\
& = z\Phi(x)(m\otss T_n).
\end{align*}

One can show very similarly that, for any $a\in A$, we have  \[\beta_z(\pi(a))(m\otss T_n)= \pi(a) (m\otss T_n),\] since $\pi(a)(m\otss T_n)\in M\otss\OY^{n}$, for any $n\in\ZZ$. 
\end{proof}

The Gauge Invariant Uniqueness Theorem now gives us the following result. 

\begin{theorem}\label{gaugeresult}

Let \tn{[\AMB, $U_M$]}: \AXA\ $\rightarrow$ \BYB\  be a morphism in \cat , and let the pair $(\pi, \Phi)$ be the C-covariant representation of \AXA . Then, the associated homomorphism $\sigma : \OX\rightarrow \KK(M\otss \OY)$ is injective. Moreover, the Cuntz-Pimsner algebra $\OX$ is isomorphic to the \C-algebra \C$(\pi,\Phi)\subset\KK(M\otss\OY).$

\end{theorem}

\begin{corollary}
If \AXA\ is a nondegenerate subcorrespondence of \BYB , then $\OX$ is isomorphic to a subalgebra of $\OY$. 
\end{corollary}

\begin{proof}
Follows from Proposition \ref{proplast} and Theorem \ref{gaugeresult}.
\end{proof}

When \AXA\ is a regular \C-correspondence we may view \big[$\pre AX_A^{\otimes n}$, $1_{X^{\otimes (n+1)}}$\big] as a morphism from \AXA\ to \AXA\ in \cat . Let $(\pi, \Phi)$ be the associated \ec-covariant representation, and $(\Upsilon,t)$ be the universal covariant representation of \AXA . Then, 
the homomorphism $\sigma :  \OX\rightarrow \KK(X^{\otimes n} \ots \OX)$ defines a left action of $\OX$ on $(X^{\otimes n} \ots \OX).$ Now, let $S\in\OX$, and let $y=(y_1\ots y_2\ots...\ots y_n)\in X^{\otimes n}.$ By the construction of $(\pi, \Phi)$ we have 
\begin{align*}
\Phi(x)(y\ots S)&=(x\ots y_1\ots....\ots y_{n-1})\ots t(y_n)S \\
\pi(a)&= \varphi_n(a) \ot 1_{\OX},
\end{align*}
for any $a\in A$, and $x\in X,$ where $\varphi_n(a)$ denotes the operator $\varphi_X(a) \ot  1_{X^{\otimes n-1}} \in \LL(X^{\otimes n})$. We now have the following Proposition.

\begin{proposition}\label{regularcuntziso}
For a regular \C-correspondence \AXA , and for $n>0$, we have the isomorphism $X^{\otimes n} \ots \OX \cong \OX$ as $\OX-\OX$ correspondences. 
\end{proposition}

\begin{proof} Let $n>0.$  Let $(\Upsilon,t)$ denote the universal covariant representation of \AXA , and let $(\pi, \Phi)$ be the \ec-covariant representation of \AXA\ on $\KK(X^{\otimes n}\ots \OX)$. Lemma \ref{isos} (iii) implies that the map $U:  X^{\otimes n}\ots \OX \rightarrow \OX $ determined on elementary tensors by $z\ots S \mapsto t^n(z)S$ is a Hilbert $A-\OX$ module isomorphism. Since $t(x)$ and $\Upsilon(a)$ generates $\OX$, it suffices to show 
\[ U( t(x)\cdot \xi ) =t(x)U(\xi), \midtext{and} U( \Upsilon(a) \cdot  \xi)= \Upsilon(a)U(\xi),\]
for any $\xi\in (X^{\otimes n}\ots \OX). $ Let $y_i = y_{i,1}\ots y_{i,2}\ots...\ots y_{i,n}\in X^{\otimes n},$ $S_i\in\OX$ for any $i\in F\subset \NN$ finite.  Then we have 
\begin{align*}
t(x)\cdot \sum_{i\in F}(y_i \ots S_i) & = \sigma_X(t(x))\sum_{i\in F}(y_i \ots S_i) \\
&= \Phi(x)\sum_{i\in F}(y_i \ots S_i) \\
&= \sum_{i\in F}\Phi(x)(y_{i,1}\ots y_{i,2}\ots...\ots y_{i,n} \ots S_i) \\
&= \sum_{i\in F} x\ots y_{i,1}\ots y_{i,2}\ots...\ots y_{i,n-1}\ots t(y_{i,n})S_i. 
\end{align*}
This implies 
\begin{align*}
U\Big(t(x)\cdot \sum_{i\in F}(y_i \ots S_i) \Big) &= \sum_{i\in F}t^n(x\ots y_{i,1}\ots y_{i,2}\ots...\ots y_{i,n-1})t(y_{i,n})S_i\\
&=\tsum_{i\in F}t(x)t^{n-1}(y_{i,1}\ots y_{i,2}\ots...\ots y_{i,n-1})t(y_{i,n})S_i\\
&=t(x)\sum_{i\in F}t^n(y_{i,1}\ots y_{i,2}\ots...\ots y_{i,n}\ots S_i)\\
&=t(x)\cdot U\Big(\sum_{i\in F}(y_i\ots S_i)\Big).
\end{align*}
Similarly, for $a\in A$, we have 
\begin{align*}
\Upsilon(a)\cdot \sum_{i\in F}(y_i \ots S_i) &= \sigma_X(\Upsilon(a))\sum_{i\in F}(y_i \ots S_i) \\
&= \pi(a)\sum_{i\in F}(y_i \ots S_i) \\
&= \sum_{i\in F}\varphi_n(a)y_i \ots S_i.
\end{align*}
This implies that 
\begin{align*}
U\Big(\Upsilon(a)\cdot \tsum_{i\in F}(y_i \ots S_i) \Big)&= U\Big(\sum_{i\in F}\varphi_n(a)y_i \ots S_i\Big)\\
&=\sum_{i\in F}  t^n(\varphi_n(a)y_i)S_i \\
&= \sum_{i\in F} \Upsilon(a)t^n(y_i)S_i \\
&= \Upsilon(a)\cdot U\Big(\sum_{i\in F}(y_i \ots S_i) \Big),
\end{align*}
which completes the proof.

\end{proof}

\section{The Functor}

\begin{theorem}\label{main}

Let \tn{[\AMB, $U_M$]}: \AXA $\rightarrow$ \BYB\  be a morphism in \cat . Then the assignments \AXA $\mapsto \OX$ on objects and \[  [\pre AM_B, U_M] \mapsto [\pre {\OX}(M\otss \OY)_{\OY}] \]
on morphisms define a functor $\fun$ from \cat\ to the enchilada category. 

\end{theorem}

\begin{proof} Let \tn{[\AMB, $U_M$]}: \AXA $\rightarrow$ \BYB\ , and \tn{[\BNC, $U_N$]}: \BYB $\rightarrow$ \CZC\  be morphisms in \cat . We want to show  \begin{itemize}
\item $\fun\left( [\pre A(M\otss N)_C, U_{M\otss N}] \right) = \fun$([\AMB, $U_M$])${\otimes}_{\OY}\fun$([\BNC, $U_N$]), and
\item $\fun([\pre AA_A, U_A]) = [\pre {O_X}(O_X)_{\OX}]$.
\end{itemize}
We start with proving the isomorphism
\begin{equation}\label{composition}
\pre {\OX}(M\otss \OY){\otimes}_{\OY} (N\otsc \OZ)_{\OZ} \cong \pre {\OX}(M\otss N\otsc \OZ)_{\OZ}. 
\end{equation}
Let $(\pi_1, \Phi_1)$, and  $(\pi_2, \Phi_2)$ be the \ec-covariant representations of \AXA\ , as in Definition \ref{rep}, on $\KK(M\otss \OY)$ and $\KK(M\otss N \otsc \OZ)$, respectively. Let $(\pi, \Phi)$ be the \ec-covariant representation of \BYB\ on $\KK(N\otsc\OZ).$ We already have the Hilbert $A-\OZ$ correspondence isomorphism 
\[ U: (M\otss \OY)\oy (N\otsc \OZ)  \rightarrow (M\otss N \otsc \OZ), \]
which gives rise to the isomorphism
\[ \ad U: \LL\big((M\otss \OY)\oy (N\otsc \OZ)\big) \rightarrow \LL\big(M\otss N \otsc \OZ\big).\]
Therefore, it suffices to show $U$ preserves the left $\OX$-module structure. Since $U$ preserves the left action of $A$, for any $a\in A$ we observe
\begin{align*}
\ad U\big(\pi_1(a)\ot 1_{N\otsc \OZ}\big) &= \ad U\big( (\varphi_M(a)\ot 1_{\OY}) \ot 1_{N\otsc \OZ}\big)\\
& = \varphi_M(a)\ot 1_N \ot 1_{\OZ} \\
&= \pi_2(a). 
\end{align*} By following the construction of $\Phi_1(x)$ and $\Phi_2(x),$ we next show similarly that \[\ad U(\Phi_1(x)\ot1_{N\otsc \OZ} ) = \Phi_2(x).\] Let $t_Y$ and $t_Z$ be the linear maps associated to the universal covariant representations of \BYB\ and \CZC\ , respectively. Consider the isomorphisms
\begin{align*}
&U_M: X\ots M \rightarrow M\otss Y, & & U_N: Y\otss N \rightarrow N\otsc Z, \\
&V_Y: Y\otss \OY\rightarrow \overline{t_Y(Y)\OY}, & & V_Z: Z\otsc \OZ\rightarrow \overline{t_Z(Z)\OZ},
\end{align*}
\[U_{M\otss N}: X\ots M\otss N \rightarrow M\otss N\otsc Z .\]
For $x\in X,$ $y\in Y$, we have the linear maps 
\begin{align*}
&T_1(x): M\rightarrow M\otss Y, &m&\mapsto U_M(x\ots m)\\
&T(y): N\rightarrow N\otsc Z, &n&\mapsto U_N(y\otss n)\\
&T_2(x): M\otss N \rightarrow M\otss N\otsc Z, &\nu&\mapsto U_{M\otss N}(x\ots\nu).
\end{align*}
Notice that for $x\in X$, $m\in M$ and $n\in N$ we have
\begin{align*}
T_2(x)(m\otss n) & = U_{M\otss N}(x\ots m \otss n)\\
&=(1_M\ot U_N)[U_M(x\ots m)\otss n] \\
&= (1_M\ot U_N)(T_1(x)\ot 1_N)(m\otss n),
\end{align*}
which implies that
 \begin{align}
\Phi_2(x) &=  (1_{M\otss N}\ot V_Z)(T_2(x)\ot 1_{\OZ}) \nonumber \\
 &=(1_{M\otss N}\ot V_Z)(1_M\ot U_N \ot 1_{\OZ})(T_1(x)\ot 1_N\ot 1_{\OZ}).\label{phi2}
 \end{align}
On the other hand, recall that $\Phi_1(x)$ and $\Phi(y)$ are defined by \begin{equation}\label{eqs}\Phi_1(x)= (1_M\ot V_Y)(T_1(x)\ot 1_{\OY}), \midtext{and} \Phi(y)=(1_N\ot V_Z)(T(y)\ot 1_{\OZ}).
\end{equation}  We aim to prove the equality
\begin{equation}\label{maineq}
\Phi_2(x)U=U(\Phi_1(x) \ot 1_{N\otsc \OZ}). 
\end{equation}
Let  $\iota_{\OY}$ be the isomorphism
\[\iota_{\OY}:\pre B(\OY)\oy (N\otsc \OZ)_{\OZ} \rightarrow  \pre B(N\otsc\OZ)_{\OZ}\] determined by  $S\oy \nu \mapsto \sigma_Y(S)\nu$, for $S\in \OY$, $\nu\in N\otsc \OZ$, where $\sigma_Y$ denotes the left action of $\OY$ on the Hilbert module $N\otsc\OY$. We first claim 
\begin{equation}\label{claim}
(1_M\ot 1_N\ot V_Z)(1_M\ot U_N\ot 1_{\OZ})(1_M\ot 1_{Y}\ot \iota_{\OY})=U(1_M\ot V_Y\ot 1_{N\otsc \OZ}).
\end{equation}
It suffices to check equality (\ref{claim}) for the elements of form $(m\otss y\otss S)\oy\nu$, where $m\in M,$ $y\in Y,$ $S\in\OY,$ $\nu\in (N\otsc\OZ)$: since $V_Y(y\otss S)=t_Y(y)S$ we have
\allowdisplaybreaks
\begin{align*}
&U(1_M\ot V_Y\ot 1_{N\otsc \OZ})(m\otss y\otss S\oy\nu)\\
&= U(m\otss t_Y(y)S\oy\nu) \\
&=m\otss\sigma_Y(t_Y(y)S)\nu \\
&=m\otss \Phi(y)\sigma_Y(S)\nu  & \text{($\sigma_Y(t_Y(y))=\Phi(y)$)}\\
&= (1_M\ot 1_N\ot V_Z)(1_M\ot T(y)\ot 1_{\OZ})(m\otss \sigma_Y(S)\nu) & \text{(\ref{eqs})}
\end{align*}
Since by construction we have $\big(T(y)\ot 1_{\OZ}\big)(\xi)= \big(U_N\ot 1_{\OZ}\big)(y\otsc \xi),$ for any $y\in Y,$ $\xi\in N\otsc\OZ$, we may continue our computation as
\begin{align*}
&= (1_M\ot 1_N\ot V_Z)(1_M\ot U_N\ot 1_{\OZ})(m\otss y\otss \sigma_Y(S)\nu)\\
&= (1_M\ot 1_N\ot V_Z)(1_M\ot U_N\ot 1_{\OZ})(1_M\ot 1_{Y}\ot \iota_{\OY})(m\otss y\oy S\oy \nu),
\end{align*}
which completes the proof of our claim. 

We are now ready to prove equality \ref{maineq}.  Once again let  $m\in M,$ $S\in\OY,$ $\nu\in (N\otsc\OZ)$. We have 
\begin{align*}
&\Phi_2(x)U(m\otss S\oy \nu)\\
&\myeq(1_M\ot 1_N\ot V_Z)(1_M\ot U_N\ot 1_{\OZ})(T_1(x)\ot 1_N\ot 1_{\OZ})U(m\otss S\oy \nu)\\
&= (1_M\ot 1_N\ot V_Z)(1_M\ot U_N\ot 1_{\OZ})(T_1(x)\ot 1_N\ot 1_{\OZ})[m\otss \sigma_Y(S)\nu] \\
&= (1_M\ot 1_N\ot V_Z)(1_M\ot U_N\ot 1_{\OZ})[U_M(x\ots m)\otss  \sigma_Y(S)\nu]\\
&= (1_M\ot 1_N\ot V_Z)(1_M\ot U_N\ot 1_{\OZ})(1_M\ot 1_{Y}\ot \iota_{\OY})(U_M(x\ots m)\otss S\oy\nu)\\
&\myothereq U(1_M\ot V_Y\ot 1_{N\otsc \OZ})(U_M(x\ots m)\otss S\oy\nu) \\
&= U(1_M\ot V_Y\ot 1_{N\otsc \OZ})(T_1(x)\ot 1_{\OY}\ot 1_{N\otsc \OZ})(m\otss S\oy\nu)\\
&= U(\Phi_1(x)\ot 1_{N\otsc\OZ})(m\otss S\oy\nu),
\end{align*}
as desired. The equality (*) is followed by (\ref{phi2}), and the equality (**) is followed by (\ref{claim}).

Now we have the following diagram.
 
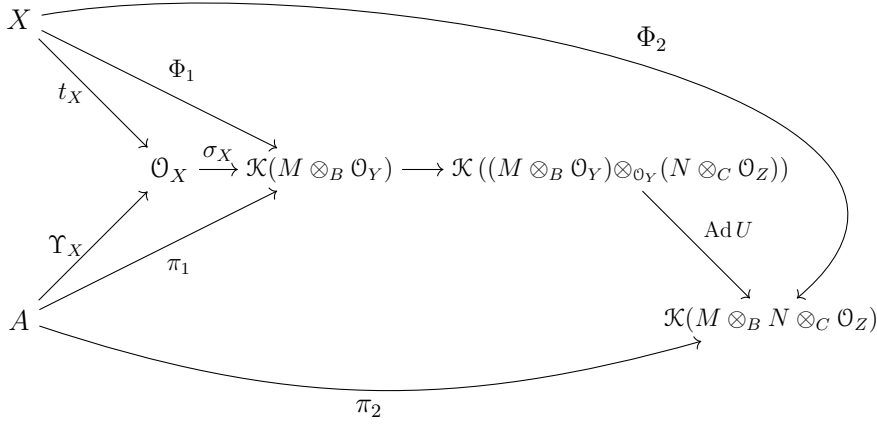
\begin{center}
\begin{tikzpicture}[node distance=2cm, scale=1, transform shape]
\node (O)  [scale=0.9]  {$\OX$};
\node (A) [left of=O, below of=O, scale=0.9] {$A$};
\node (X) [left of=O, above of=O, scale=0.9] {$X$};
\draw[->] (X) to node[scale=0.8, left] {$t_X$} (O);
\draw[->] (A) to node[scale=0.8, left]{$\Upsilon_X$} (O);
\node (L) [right of=O, scale=0.8]{$\KK(M\otss \OY)$};
\draw[->] (O) to node[scale=0.8]  {$\sigma_X$} (L);
\node (LL) [right of=L, right of=L, scale=0.8]{$\KK\left( (M\otss \OY){\otimes}_{\OY} (N\otsc \OZ) \right)$};
\draw[->] (L) to  (LL);
\draw[->] (X) to node[scale=0.8] {$\Phi_1$} (L);
\draw[->] (A) to node[swap, scale=0.8] {$\pi_1$} (L);
\node (LLL) [right of=LL, below of=LL, scale=0.8]{$\KK(M\otss N\otsc \OZ)$}; 
\draw[->] (LL) to node[scale=0.7]{$\ad U$} (LLL);
\draw[->, out=10, in=40] (X) to node[scale=0.9] {$\Phi_2$} (LLL);
\draw[->, bend right=0.6cm] (A) to node[scale=0.9, swap] {$\pi_2$} (LLL);
\end{tikzpicture}
\end{center}
This means, denoting by $(\Upsilon_X,t_X)$ the Cuntz-Pimsner representation of $\OX$, we have 
\begin{itemize}
\item $\ad U(\sigma_X(\Upsilon_X(a))\ot 1_{N\otsc \OZ} ) = \ad U(\pi_1(a)\ot 1_{N\otsc \OZ}) = \pi_2(a)$
and 
\item $\ad U(\sigma_X(t_X(x))\ot 1_{N\otsc \OZ} )= \ad U(\Phi_1(x)\ot 1_{N\otsc \OZ} ) = \Phi_2(x) $
\end{itemize}
for $a\in A$, $x\in X$, which is enough to conclude that $U$ preserves  the left action of $\OX$, since the elements $\Upsilon_X(a)$ and $t_X(x)$ generate $\OX$. 

It remains to show that $\fun$ maps the identity morphism [\AAA, $U_A$]:\AXA$\rightarrow$ \AXA\ in $\cat$ to the identity morphism [$\pre {\OX}(\OX)_{\OX}$] in the enchilada category. This follows immediately from Lemma \ref{id}.
\end{proof}

\begin{remark}\label{xn}

Let \AXA\ be a regular correspondence. Then, for any n>0, we have 
\[ \fun\left(\big[X^{\otimes n}, 1_{X^{\otimes (n+1)}}\big]\right)=\big[\pre {\OX}{\OX}_{\OX}\big],\]
by Proposition \ref{regularcuntziso}.

\end{remark}

\section{Applications}
\subsection{Muhly \&\ Solel Theorem} 

Muhly and Solel introduced the notion of Morita equivalence for \C-correspondences \citep{MS} as follows: \AXA\ and \BYB\ are called Morita equivalent, denoted by \AXA$\stackrel{SME}{\sim}$\BYB , if there exists an \ibm\ \AMB\ such that 
\[ \pre A(X\ots M)_B \cong \pre A(M\otss Y)_B. \]
They proved that Morita equivalent injective \C-correspondences have Morita equivalent Cuntz-Pimsner algebras. In \citep{EKK}, the authors presented a proof for possibly non-injective \C-correspondences. In this section, we discuss how our functor provides a very practical method to  recover this result. 

First, recall that in the enchilada category  [\AMB] is an isomorphism if and only if \AMB\ is an \ibm . On the other hand, by Proposition \ref{catiso}, we have that  [$\pre AM_B, U_M$]: $\pre AX_A \rightarrow \pre BY_B$ is an isomorphism in \cat\ if and only if  \AXA\ and \BYB\ are Morita equivalent.

\begin{theorem}\label{MS}
If two \C-correspondences \AXA\ and \BYB\ are Morita equivalent, then their Cuntz-Pimsner algebras $\OX$ and $\OY$ are Morita equivalent (in the sense of Rieffel). 
\end{theorem}

\begin{proof} Since \AXA\ and \BYB\ are Morita equivalent, there exists an \ibm\ with an isomorphism $U_M: X\ots M \rightarrow M\otss Y$, which implies  [\AMB , $U_M$] is an isomorphism in \cat . This means $\fun{[\pre AM_B , U_M]}$=$[\pre {\OX}(M\otss\OY)_{\OY}]$ is an isomorphism in the enchilada category. Hence, the \C-algebras $\OX$ and $\OY$ are Morita equivalent. \end{proof}

\subsection{Cuntz-Pimsner Algebras of Shift Equivalent \C-correspondences}

In 1973 Williams introduced "elementary strong shift equivalence''  and "strong shift equivalence''  for the class of matrices with non-negative integer entries, with the goal of characterizing the topological conjugacy of subshifts of finite type \citep{Williams}. The relations were defined as follows: let $X$ and $Y$ be matrices as described. 
\begin{itemize}
\item $X$ and $Y$ are \emph{elementary strong shift equivalent}, denoted by $X \stackrel{S}{\sim} Y$,  if there exist matrices $R$ and $S$ with non-negative integer entries such that $X=RS$ and $Y=SR.$
\item The transitive closure of the relation $\stackrel{S}{\sim}$ is called \emph{strong shift equivalence}.
\end{itemize}

Putting the result of Williams \citep{Williams} and the result of Cuntz and Krieger \citep{CK} together, one concludes that strong shift equivalent matrices (with non-negative integer entries) have Morita equivalent Cuntz-Krieger algebras. In \citep{MPT}, Muhly, Pask and Tomforde formulated this in the setting of \C-correspondences as follows.

\begin{definition}\normalfont
Two \C-correspondences \AXA\ and \BYB\ are called \emph{elementary strong shift equivalent}, denoted by \AXA$\stackrel{S}{\sim}$\BYB\ , if there are \C-correspondences \ARB\ and \BSA\ such that 
\[ X \cong R\otss S \midtext {and} Y \cong S\ots R \]
as \C-correspondences.

\C-correspondences \AXA\ and \BYB\ are called \emph{strong shift equivalent}, denoted by \AXA$\stackrel{SSE}{\sim}$\BYB , if there are \C-correspondences $\{Z_i\}_{\{0\leq i \leq n\}}$ such that $Z_0=X$, $Z_n=Y$, and $Z_i\stackrel{S}{\sim}Z_{i+1},$ for each $i$.

\end{definition} 

\begin{theorem}[\cite{MPT}]\label{MPT}

If two regular  \C-correspondences \AXA\ and \BYB\ are strong shift equivalent, then their Cuntz-Pimsner algebras $\OX$ and $\OY$ are Morita equivalent. 
\end{theorem}

\begin{proof} 
Let \AXA\ and \BYB\ be elementary strong shift equivalent. Then, there exists correspondences \ARB\ and \BSA\ with the isomorphisms
\[ \phi_X: X \rightarrow  R\otss S \midtext{and} \phi_Y: Y\rightarrow S\ots R. \]
Define  
\[U_R = (1_R\ot \phi_Y^{-1} )(\phi_X\ot 1_R) \midtext{and} U_S :=(1_S\ot \phi_X^{-1})(\phi_Y\ot 1_S). \]
Notice that we have  
\[(1_R\ot U_S)(U_R\ot 1_S) = (\phi_X\ot \phi_X^{-1}). \]
This allows us to see that the diagram
\[
\begin{tikzcd}
X\ots (R\otss S) \arrow{r}{1_X\ot \phi_X^{-1}} \arrow[swap]{d}{(1_R \ot U_S)(U_R\ot 1_S)} & X\ots X \arrow{d}{1_{X^{\otimes 2}}} \\
(R\otss S)\ots X \arrow{r}\arrow{r}{\phi_X^{-1} \ot 1_X} & X\ots X
\end{tikzcd}
\]
commutes, which implies the equality $\big[ R\otss S, U_{R\otss S} \big] = \big[X, 1_{X^{\otimes 2}}\big].$  

We now show that 
$\fun\left([S, U_S]\right)$ and $\fun\left([R, U_R]\right)$ are inverses of each other: 
\begin{align*}
\fun\left([S, U_S]\right)\circ  \fun\left([R, U_R]\right) &= \fun\left([S, U_S] \circ [R, U_R] \right) \\
&= \fun\left([R\otss S, U_{R\otss S}]\right)\\
&= \fun\left([X, 1_{X^{\otimes 2}}]\right)\\
&= [\pre {\mathcal{O}_X}{\mathcal{O}_X}_{\OX}],
\end{align*}
where the last step follows by Remark \ref{xn}. It can be seen similarly that 
\[ \fun\left([R, U_R]\right )\circ \fun\left([S, U_S] \right) =  [\pre {\OY}{\OY}_{\OY}].\] Hence, the correspondences $\pre {\OX}(R\otss \OY)_{\OY}$ and $\pre  {\OX}(S\ots \OX)_{\OX}$ are inverses of each other.  \end{proof}

\begin{corollary}\label{ss}
If \AXA\ and \BYB\ are regular elementary strong shift equivalent \C-correspondences via \ARB\ and \BSA , then the injective homomorphisms
\[ \sigma_Y: \OY\rightarrow \KK(S\ots\OX) \midtext{and} \sigma_X: \OX\rightarrow \KK(R\otss \OY)\]
are surjective. 
\end{corollary}

\begin{example} Let \AXA\ be a regular \C-correspondence. Then, the \C-correspondence $\pre {\KK(X)}\big(X\ots \KK(X)\big)_{\KK(X)}$ is regular, as well. Moreover, we have 
\[ \tn{\AXA\ $\stackrel{S}{\sim}$  $\pre {\KK(X)}\big(X\ots \KK(X)\big)_{\KK(X)}$ } \]
via $\pre {\KK(X)}X_A$ and $\pre A\KK(X)_{\KK(X)}$.  Denote $X\ots \KK(X)$ by $\XX$. We know by Theorem \ref{MPT} that $\OX$ and $\OXXX$ are Morita equivalent. However, Corollary \ref{ss} implies a stronger relation between these \C-algebras:
\[\OX \cong \KK\big(\KK(X){\otimes}_{\KK(X)} \OXXX\big)\cong \OXXX , \]
where the latter isomorphism is the natural \C-algebra isomorphism as described in  the proof of Proposition \ref{proplast}. 
\end{example}

Let $G$ be a locally compact group with $\alpha: G\curvearrowright A$ and $\beta:  G\curvearrowright B$. 
An  \emph{$\alpha-\beta$ compatible action}  $\gamma$ of $G$ on \AXB\ is a homomorphism of $G$ into the group of invertible linear maps on $X$ such that 
\vspace{.2cm}
\begin{enumerate}[(i)]
\item $\gamma_s(a\cdot x)=\alpha_s(a)\cdot \gamma_s(x)$
\vspace{.2cm}
\item $\gamma_s(x\cdot b)=\gamma_s(x)\cdot \beta_s(b)$
\vspace{.2cm}
\item $\< \gamma_s(x), \gamma_s(y)\>_B = \beta_s(\<x,y\>_B)$
\end{enumerate}
\vspace{.2cm}
for each $s\in G,$ $a\in A,$ $x\in X,$ and $b\in B$; and such that each map $s\mapsto \gamma_s(x)$ is continuous from $G$ into $X$.

\begin{definition}
Let $(\pi_1, t_1)$ be a representation of \AXA\ that admits a gauge action $\alpha$, and let $(\pi_2, t_2)$ be a representation of \BYB\ that admits a gauge action $\beta$. Let $M$ be a  \C$(\pi_1, t_1)$-\C$(\pi_2, t_2)$ \ibm\ .
The Morita equivalence between \C$(\pi_1, t_1)$ and \C$(\pi_2, t_2)$ is called \emph{gauge equivariant} if there exists an $\alpha-\beta$ compatible action of $\mathbb{T}$ on $M$. 
\end{definition}

\begin{theorem}\label{equi}
The Morita equivalence in Theorem \ref{MPT} is gauge equivariant.
\end{theorem}

\begin{proof}

Let \AXA\ and \BYB\ be regular elementary strong shift equivalent \C-correspondences via \ARB\ and \BSA . Denote  the universal covariant representation of \AXA\ by $(\Upsilon, t)$, and the \ec-covariant representation on $\KK(R\otss \OY)$ by $(\pi, \Phi).$ By Corollary \ref{ss} we have an isomorphism $\sigma: \OX \rightarrow \KK(R\otss\OY)$ such that \[\sigma(t(x))=\Phi(x) \midtext{and} \sigma(\Upsilon(a))=\pi(a)\]
for any $x\in X$, $a\in A$, which allows us to view  $R\otss\OY$ as an imprimitivity $\OX-\OY$ bimodule. Now,  denote by $\alpha$ the gauge action for $\OX$ and by $\gamma$ the gauge action for $\OY$. We show that the homomorphism $z\mapsto 1_R \ot \gamma_z $ is an $\alpha-\gamma$ compatible action of $\TT$ on the \ibm\  $\pre {\OX}(R\otss\OY)_{\OY}$. To this end, we first prove the equality
\begin{equation}\label{eq}
(1_R\ot\gamma_z)[T\cdot \xi ] = \alpha_z(T)\cdot (1_R\ot\gamma_z)(\xi) 
\end{equation}
for any $T\in\OX$, $\xi\in R\otss\OY$. Let $x\in X$ and $a\in A$. It suffices to  let $T=t(x)$ and $T=\Upsilon(a)$ as such elements generate $\OX$. For $r\in R$ and $S_n\in\On$ we have 
\begin{align*}
\alpha_z(t(x))\cdot (\rz)(r\otss S_n) &= (zt(x))\cdot [z^n (r\otss S_n)]\\
&= \sigma(z t(x))[z^n (r\otss S_n)]\\
& = z^{n+1}\Phi(x)(r\otss S_n)\\
& = (1_R\ot\gamma_z)[t(x)\cdot (r\otss S_n)],
\end{align*}
where the last step follows from (\ref{g}). One can verify (\ref{eq}) for $T=\Upsilon(a)$, very similarly.

Next, we show
\[\< (\rz)\xi, (\rz)\nu \>_{\OY} = \gamma_z\big(\<\xi, \nu\>_{\OY}\big),\]
for $\xi , \nu\in (R\otss \OY)$. Let $r'\in R$, $S_m\in \OY^m$.  We have 
\begin{align*}
\< (\rz)(r\otss S_n), (\rz)(r'\otss S_m)\>_{\OY} &= z^{m-n}\<r\otss S_n, r'\otss S_m\>_{\OY}\\
&= z^{m-n} S_n^*\cdot \<r,r'\>_B \cdot S_m \\
&= \gamma_z \big( S_n^*\cdot \<r,r'\>_B \cdot S_m \big) \\
& = \gamma_z \big( \< r\otss S_n, r'\otss S_m\>_{\OY} \big),
\end{align*}
which completes the proof since elements of form $r\otss S_n$ densely span $R\otss\OY$.
\end{proof}

\subsection{Pimsner Dilations}

For an injective \C-correspondence \AXA , one can construct a Hilbert bimodule that contains a copy of $X$ as a subspace. The Pimsner dilation $\X$, which was first introduced by Pimsner \citep{pimsner}, is the minimal Hilbert bimodule that contains  \AXA\ as a sub-correspondence \cite[Theorem 3.5]{KK}. To describe Pimsner dilations, we use Katsura's so-called \emph{cores}. The detailed information about these particular \C-algebras can be found in \citep{katsura}; here we give a quick review.

For each $n\in \NN$ set $\B_n = \Psi_{t^n}(\KK(X^{\otimes n }) )\subset$ \C$(\pi, t).$ 
Note that $\B_0 := \pi(A)$ and that $\B_n \cong \KK(X^{\otimes n })$ when $(\pi, t )$ is injective.  For $m, n \in \NN$ with $m\leq n$, define $\B_{[m,n]}\subset$ \C$(\pi,t)$ by
\[
 \B_{[m,n]} = \B_m + \B_{m+1} + .... + \B_n .
 \]
We denote $\B_{[n,n]}$ by  $\B_n$ for $n\in \NN .$ All $\B_{[m,n]}$'s are \C-subalgebras of \C$(\pi,t).$ In addition, $\B_{[k,n]}$ is an ideal of $\B_{[m,n]}$ for $m, k, n\in \NN$ with $m\leq k\leq n. $ In particular, $\B_n$ is an ideal of $\B_{[0,n]}$ for each $n\in \NN .$  For $m\in \NN,$ define the  \C-subalgebra $\B_{[m,\infty)}$ of \C$(\pi,t)$ by
\[
\B_{[m,\infty)} = \overline{\bigcup_{n=m}^{\infty} \B_{[m,n]}}. 
\]
Notice that $\B_{[m,\infty)}$ is an inductive limit of the increasing sequence of \C-algebras $ \{\B_{[m,n]} \}_{n=m}^{\infty}. $ The \C-algebra $\core$ is called the $core$ of the \C-algebra  \C$(\pi,t).$ The core $\core$ naturally arises when  \C$(\pi,t)$ admits gauge action $\beta$, and it coincides with the fixed point algebra \C$(\pi,t)^{\beta}.$

Now, the Pimsner dilation is defined as follows: let $(\Upsilon, t)$ be the universal covariant representation of an injective \C-correspondence \AXA . Then 
\begin{equation}\label{XO}
\X:=\overline{t(X)\core} = \clspn\{t(x)k : x\in X, k\in \core\}\
\end{equation}
is a subspace of $\OX .$ We may define right and left actions of $\core$ on $\X$ simply by multiplication.  Notice that for any $\nu, \xi \in \X ,$ we have $\<\nu,\xi\>_{\OX} = \nu^*\xi \in \core.$ Moreover, we observe that
\[\pre {\OX}\<\X,\X\> = \X\X^*= \overline{t(X)\core t(X)^*}=\cores ,\]
and thus $\X$ can be viewed as a \C-correspondence over $\core$ such that  the left action homomorphism $\varphi_{\X}:\core \rightarrow \LL(\X)$ is an isomorphism onto $\KK(\X).$ 

\begin{lemma}\label{JX} Let \AXA\ be an injective \C-correspondence with the universal covariant  representation $(\Upsilon,t).$ Then we have the following. 
\begin{enumerate}[\normalfont(i)] 
\item The Hilbert $\core$-modules $(X\ots\core )$ and $\X$ are isomorphic. 
\item $J_{\X} = \cores$.
\item The isomorphism class $[\pre A{\core}_{\core}]$ is a morphism \AXA $\rightarrow \pre {\core}{\X}_{\core}$ in \cat .  
\end{enumerate}
\end{lemma}

\begin{proof} It is straightforward to verify that the map $X\odot \core \rightarrow \overline{t(X)\core}$ determined on elementary tensors by $x\otimes S\mapsto t(x)S$  preserves the left-module structure and the semi-inner product. Moreover, it is surjective. Hence, it extends to a Hilbert module isomorphism $(X\ots \core) \rightarrow \overline{t(X)\core}$. 

Item (ii) follows from the fact that  $\varphi_{\X}:\core \rightarrow \LL(\X)$ is an isomorphism onto $\KK(\X).$ And, item (iii) follows from Lemma \ref{XinXB}, since \AXA\ is a nondegenerate subcorrespondence of $\pre {\core}{\X}_{\core}$. 
\end{proof}

Lemma \ref{JX} implies that any injective \C-correspondence \AXA\ has a  \ec-covariant  representation $(\pi,\Phi)$ on $\KK(\core\otsb \OXX).$ The \C-algebras $\OX$ and $\OXX$ are isomorphic  \citep[Theorem 2.5]{pimsner}, \citep[Theorem 6.6]{kk3}.  Then Corollary \ref{gaugeresult} tells us that $\OXX$ is nothing but the \C-algebra generated by the \ec-covariant representation of \AXA .  In the next theorem, we present an alternative proof for the isomorphism $\OX\cong \OXX$ by using the \ec-covariant representation $(\pi,\Phi)$. The proof shows the exact relation between the generators of $\OX$, $\OXX$ and \C$(\pi,\Phi)$.

\begin{theorem}
Let \AXA\ be an injective \C-correspondence with the universal covariant representation $(\Upsilon, t)$, and let $\pre \core\X_{\core}$ be as in (\ref{XO}). Denote the associated  C-covariant representation on $\KK(\core\otsb \OXX)$ by $(\pi, \Phi)$. Then, we have the isomorphisms $\OX\cong C^*(\pi, \Phi) \cong \OXX$. 
\end{theorem} 

\begin{proof} Let $(\Upsilon_{\X}, T)$ be the universal covariant  representation of  $\X$. Let $\iota$ denote the isomorphism 
\[ \KK\big(\core\otsb\OXX\big)\rightarrow \KK(\OXX)\rightarrow \OXX .\]
Proposition \ref{proplast} gives us
\[ T(t(x))=\iota(\Phi(x)) \midtext{and} \Upsilon_{\X}(\Upsilon(a))=\iota(\pi(a))\]
for any $x\in X,$ $a\in A. $ We now have the following diagram: 

\begin{figure}[h!]
\centering
\begin{tikzpicture}[scale=1.2, transform shape]
\node (O) [scale=0.8]  {$\OX$};
\node (A) [below of=O, left of=O, scale=0.8] {$A$};
\node (X) [left of=O, above of=O, scale=0.8] {$X$};
\draw[->] (X) to node[scale=0.8, swap] {$t$} (O);
\draw[->] (A) to node[scale=0.8]{$\Upsilon$} (O);
\node (L) [right of=O, right of=O, scale=0.8]{$\KK\big(\core\otsb \OXX \big)$};
\draw[->] (O) to node[scale=0.7]  {$\sigma$} (L);
\node (LL) [right of=L, right of=L, scale=0.8]{$\OXX$};
\node (XX) [right of=X, right of=X, scale=0.8]{$\X$};
\draw[right hook ->] (X) to (XX);
\draw[->] (L) to node[scale=0.7]  {$\iota$}  (LL);
\draw[->] (XX) to node[scale=0.7] {$T$} (LL);
\draw[->] (X) to node[scale=0.7] {$\Phi$} (L);
\draw[->] (A) to node[swap, scale=0.7] {$\pi$} (L);
\node (AA) [right of=A, right of=A, scale=0.8]{$\core$};
\draw[->] (AA) to node[scale=0.7, swap]{$\Upsilon_{\X}$} (LL);
\draw[right hook ->] (A) to (AA);
\end{tikzpicture}
\end{figure}

Since $\Upsilon_{\X}(\core)$ and $T(\X)$ generate $\OXX$, it suffices to show $\Upsilon_{\X}(k)\in  \iota(C^*(\pi, \Phi))$ and $T(\xi)\in  \iota(C^*(\pi, \Phi))$ for any $k\in\core$, $\xi \in \X$. First, recall that $\varphi_{\X}: \core\rightarrow \LL(\X)$ denotes the left action of $\core$ on $\X$, and $J_{\X}=\cores$. Now, for any $\xi\in \X $ and $x,y\in X$, we have 
\begin{align*}
 \varphi_{\X}\left(\Psi_t(\theta_{x,y})\right)(\xi) = \Psi_t(\theta_{x,y})\xi &= t(x)t(y)^*\xi \\
 &= t(x)\<t(y),\xi\>_{\core} \\
 &= \theta_{t(x),t(y)}(\xi) .
 \end{align*} 
Since $(\Upsilon_{\X}, T)$ is covariant, this allows us to observe that
\begin{align*}
\Upsilon_{\X}\left(\Psi_t(\theta_{x,y})\right) &= \Psi_{T}(\theta_{t(x),t(y)}) \\
& = T\left(t(x)\right)T\left(t(y)\right)^*\\
&= \iota(\Phi(x)\Phi(y)^*)\\
&=\iota(\Psi_{\Phi}(\theta_{x,y})) \in \iota(C^*(\pi, \Phi)).
\end{align*}
Very similarly, for  $k:=\theta_{x_1\ots x_2, y_1\ots y_2}\in \KK(X^{\otimes 2})$ we have, 
\begin{align*}
\varphi_{\X}\left(\Psi_{t^2}(k)\right)(\xi) &= t^2(x_1\ots x_2)t^2(y_1\ots y_2)^*\xi \\
&= t(x_1)t(x_2)t(y_2)^*t(y_1)^*\xi \\
&= \theta_{t(x_1), t(y_1)\Psi_t(\theta_{y_2, x_2})}(\xi) .
\end{align*}
Therefore, we obtain  \[\Upsilon_{\X}(k)= \Psi_{T}(\varphi_{\X}(k) ) = T(t(x_1))T[t(y_1)\Psi_t(\theta_{y_2,x_2})]^*= T(t(x_1))\Upsilon_{\X}(\Psi_t(\theta_{x_2, y_2})) T(t(y_1))^*.\]
This computation allows one to conclude that $\Upsilon_{\X}(k)\in \iota( $\C$(\pi, \Phi)),$ for all $k\in \cores$. 

On the other hand, if $k\in\B_0$, we have $k=\Upsilon(a)$ for some $a\in A$. Thus, $\Upsilon_{\X}(k)=\Upsilon_{\X}(\Upsilon(a))=\iota(\pi(a))\in  \iota(C^*(\pi, \Phi))$. 

To sum up, for any $x\in X$ and $k\in\core$, we have  \[T(t(x)k)=T(t(x))\Upsilon_{\X}(k)\in \iota(C^*(\Phi,\pi)),\] which suffices to conclude that $T(\xi)\in  \iota(C^*(\pi, \Phi))$ for any $\xi\in \X $, since elements of form $t(x)k$ densely span $\X$. 

We have shown that \C$(\Phi, \pi)\cong \OXX$. By Corollary \ref{gaugeresult} we  already have $\OX\cong C^*(\pi, \Phi)$, which completes the proof. 
\end{proof}

\section{Final Notes}
There are more applications of the functor $\mathcal{E}$ and of the \ec-covariant representations. In an upcoming paper, we use the techniques presented in this paper to study the ideals and hereditary subalgebras of Cuntz-Pimsner algebras. 

Meyer and Sehnem \cite{MeySeh} use a similar construction in the context of bicategories. We would like to note here that our development was completely independent; in fact, we were strongly motivated by the paper \cite{fanc}. However, the work of Meyer and Sehnem raises the question of what happens if we drop the injectivity condition on our morphisms in \cat . In that case, we definitely would not have  Corollary \ref{gauge}, as the representation $(\pi, \Phi)$ would not be injective. What we are not sure of is whether the injectivity condition is necessary for $\fun$ to be a functor.


\end{document}